\documentclass[12pt]{amsart}

\usepackage[margin=1.1in]{geometry}
\usepackage[T1]{fontenc}
\usepackage[utf8]{inputenc}
\usepackage{microtype}
\usepackage{amsmath,amssymb,amsthm,mathtools}
\usepackage[dvipsnames]{xcolor}
\usepackage[colorlinks=true,linkcolor=blue,citecolor=blue,urlcolor=blue]{hyperref}
\usepackage{tikz-cd} 
\usepackage{lmodern}

\usepackage{tikz}
\usepackage{wrapfig}
\usepackage{pgfplots}
\usetikzlibrary{positioning,calc,fit,backgrounds,arrows.meta}

\newcommand{\blue}[1]{{\color{blue}#1}}

\newcommand{\red}[1]{{\color{red}#1}}

\newcommand{\Bin}{\mathrm{Bin}}

\newcommand{\dist}{\mathrm{dist}}
\newcommand{\Span}{\operatorname{span}}

\newtheorem{theorem}{Theorem}
\newtheorem{proposition}{Proposition}
\newtheorem{corollary}[theorem]{Corollary}
\newtheorem{lemma}{Lemma}
\newtheorem{remark}{Remark}
\newtheorem{example}{Example}
\newtheorem{definition}{Definition}

\usepackage{microtype}
\usepackage[style=numeric-comp,hyperref=true,doi=false,url=false,sorting=none, isbn=false,giveninits=false, block=none,backend=bibtex,maxnames=99]{biblatex}
\renewbibmacro{in:}{}

\bibliography{graphonstab.bib}

\usepackage[font=small]{caption}

\newcommand{\aff}{\operatorname{aff}}
\newcommand{\R}{\mathbb{R}}

\newcommand{\x}{\mathbf{x}}

\newcommand{\Z}{\mathbb{Z}}

\newcommand{\E}{\mathbb{E}}

\newcommand{\conv}{\operatorname{conv}}

\newcommand{\bfo}{{\bf 1}}

\newcommand{\rint}{\operatorname{relint}}
\newcommand{\supp}{\operatorname{supp}}
\newcommand{\N}{\mathbb{N}}

\usepackage{tikz}
\usetikzlibrary{shapes,arrows,graphs,graphs.standard,shapes.misc}
\usetikzlibrary{matrix,decorations.pathreplacing, calc, positioning,fit}
\usetikzlibrary{shapes.geometric}
\usetikzlibrary{decorations.markings}

\newcommand{\xc}[1]{\vspace{.1cm}

\noindent {\em #1} }

\title[Structural stability for graphon systems]{Structural stability of systems and cycle covers in random graphs}
\author{Mohamed Ali Belabbas}
\address{Coordinated Science Laboratory,  Department of Electrical and Computer   Engineering and Department of Mathematics, University of Illinois at Urbana-Champaign}
\email{belabbas@illinois.edu}
\date{}
\keywords{Structural stability, Linear systems, Random graphs, Directed step-graphons, Cycle covers, Cycle polytopes}
\subjclass[2020]{Primary: 93D20, 05C80; Secondary: 05C20, 05C70, 52B12 }
\begin{document}

\maketitle

\begin{abstract}
Structural system theory studies which network topologies can sustain a prescribed system property such as controllability or stability. When the topology is itself random, the relevant question becomes probabilistic: how likely is a graph drawn from a stochastic model to sustain the property? Such probabilities measure the abundance and robustness of the property across topologies, and indicate whether systems requiring it can be reliably deployed in uncertain environments. We address this question for asymptotic stability of linear systems in the directed graphon setting.

We consider two graph-theoretic properties. The first is $\mathcal N$, and it requires that for every $k\leq n$ some $k$-vertex induced subdigraph of $D$ admits a cycle cover, and the second is $\mathcal S$, which requires that these subdigraphs can be chosen so that their node sets form a nested sequence $V_1\subset\cdots\subset V_n=V(D)$ starting from a single  vertex with a loop. We have shown that $\mathcal N$ is necessary and $\mathcal S$ is sufficient for structural stability. We sample $D$ from a directed step-graphon $W$. Our main results give necessary and sufficient conditions for $\Pr(\mathcal N)\to 1$ and $\Pr(\mathcal S)\to 1$ as $n\to\infty$.  In more detail, to a step-graphon $W$ with skeleton digraph $S$ on $q$ nodes and concentration vector $x^*$ we associate a cycle polytope $\vec{\mathcal X}(S)\subseteq\Delta_q$. The conditions are then formulated in terms of the position of $x^*$ within $\vec{\mathcal X}(S)$, the dimension of the polytope, the loop density of $W$ and, for $\mathcal S$, an ordering condition on the cycles of the skeleton. Together these results identify, for directed step-graphons, the regime in which a sampled topology is overwhelmingly likely or unlikely to sustain stable dynamics.
\end{abstract}

\section{Introduction}\label{sec:introduction}

\subsection{Background: structural system theory}\label{ssec:intro_struct}

Structural system theory asks when a given network topology can sustain a prescribed system property. Such properties include controllability and stability. In more detail, consider a network of $n$ agents with states $x_1,\dots,x_n$ whose communication topology is described by a directed graph $D=(V,\vec E)$, with vertices $v_1,\dots,v_n$ representing the agents and a directed edge $v_i \to v_j$ indicating that agent $j$ can access state information from agent $i$. Given $D$, a dynamics $\dot x(t)=f(x(t))$ is said to be {\em compatible} with $D$ if the dynamics of $x_i$ depend only on its incoming neighbors:
$$
\frac{\partial f_i}{\partial x_j}\not\equiv 0 \iff v_j\to v_i\in \vec E.
$$
We denote by $\Sigma_D$ the set of differentiable dynamics compatible with $D$ (replacing differentiability by $\mathrm C^k$, smoothness, analyticity, or linearity as the context dictates). Given a desired system property $\mathcal P$ (for instance, asymptotic stability at a given equilibrium) we say that $D$ {\em sustains} $\mathcal P$ if there exists $f\in\Sigma_D$ satisfying~$\mathcal P$. This line of research was initiated by C.-T.~Lin in his seminal paper~\cite{lin1974structural} on structural controllability, and has since seen substantial development, including work on structural stability by Belabbas and Kirkoryan~\cite{belabbas2013sparse,belabbaskirkoryancdc2014decentralized} and work on ensemble structural controllability by Chen~\cite{chen2021sparse}. These papers deal with criteria for fixed topologies; the study of random topologies was initiated for structural stability by Belabbas and Kirkoryan in~\cite{belabbas2022stable} for the case of Erd\H os--R\'enyi random graphs, and by Belabbas, Chen and Ba\c{s}ar~\cite{belabbas2023geometric,belabbas2021h} for the case of step-graphons. Controllability for specific random systems, as opposed to random structures, has been investigated by O'Rourke and Touri~\cite{o2016conjecture}, where they proved a conjecture of Godsil~\cite{godsil2012controllable}.

In the case of linear systems, the topology $D$ corresponds canonically  to a {\em sparse matrix space} (SMS) or {\em zero pattern}
$$
Z_E:=\Span\{e_{ji}:(i,j)\in \vec E\}\subseteq \R^{n\times n},
$$ where $e_{ij}$ is the canonical matrix with a $1$ in position $(i,j)$, and zeros everywhere else; equivalently, $Z_E$ is the set of matrices with a fixed $0/\ast$ pattern. See Figure~\ref{fig:HD} for an illustration of the correspondence. We say $Z_E$ (or, equivalently, $D$) is {\em stable} if it contains a Hurwitz matrix, i.e., a matrix whose eigenvalues all have strictly negative real part. Since the set of Hurwitz matrices is open, $Z_E$ contains a Hurwitz matrix if and only if it contains one whose entries are nonzero at every position allowed by the pattern; thus, for linear dynamics, $D$ sustains asymptotic stability in the sense above if and only if $Z_E$ is stable.

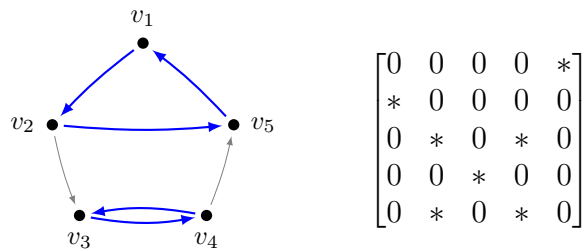
\begin{figure}
\centering
\begin{minipage}{0.1\textwidth}
\begin{tikzpicture}[scale=1.2]
\begin{scope}[xshift=-5cm]
\node [circle,fill=black,inner sep=1.5pt,label=left:{\color{black}\footnotesize $v_2$}] (2) at (0, -0.9) {};
\node [circle,fill=black,inner sep=1.5pt,label=above:{\color{black}\footnotesize $v_1$}] (1) at (1, 0) {};
\node [circle,fill=black,inner sep=1.5pt,label=right:{\color{black}\footnotesize $v_5$}] (5) at (2, -0.9) {};
\node [circle,fill=black,inner sep=1.5pt,label=below:{\color{black}\footnotesize $v_3$}] (3) at (0.3, -1.9) {};
\node [circle,fill=black,inner sep=1.5pt,label=below:{\color{black}\footnotesize $v_4$}] (4) at (1.7, -1.9) {};
\node [] at (0, -2.5) {};

\path[draw,thick,shorten >=2pt,shorten <=2pt]
(1) edge[blue,bend right=5,-latex] (2)
(2) edge[help lines,opacity=1,bend right=5,-latex] (3)
(3) edge[blue,bend right=10,-latex] (4)
(4) edge[blue,bend right=10,-latex] (3)
(4) edge[help lines,opacity=1,bend right=5,-latex] (5)
(2) edge[blue,bend right=5,-latex] (5)
(5) edge[blue,bend right=5,-latex] (1)
;
\end{scope}
\end{tikzpicture}
\end{minipage}
\hspace{3cm}
$\begin{bmatrix}
0 & 0 & 0 & 0 & \ast\\
\ast & 0 & 0 & 0 & 0\\
0 & {\ast} & 0 & \ast & 0\\
0 & 0 & \ast & 0 & 0\\
0 & \ast & 0 & {\ast} & 0
\end{bmatrix}$
\caption{The correspondence between digraphs and classes of linear systems: an edge $v_i\to v_j$ corresponds to a potentially nonzero entry in position $(j,i)$. A Hamiltonian decomposition (a set of vertex-disjoint directed cycles covering all $n$ vertices) is depicted in blue.}\label{fig:HD}
\end{figure}

Though a stable SMS contains an open set of Hurwitz matrices, this set may be {\em vanishingly small} relative to the ambient space. Illustrating this fact is the trivially stable diagonal SMS case, which contains only a $2^{-n}$ fraction of Hurwitz matrices under independent Gaussian sampling for its $\ast$-entries. Hence, sampling matrices from the pattern does not scale well computationally with $n$: finding a Hurwitz matrix certifies stability, but the diagonal case shows that an {\em exponentially large} number of samples may be needed to find one and, reciprocally, probabilistic certificates of {\em lack} of stability obtained by sampling are at least as computationally expensive.  This is in contrast with structural controllability~\cite{lin1974structural}, where the property holds either generically or not at all. This failure of probabilistic certificates motivated the search for analytic, graph-theoretic conditions for stability and led to the following pair of conditions, stated in terms of cycle covers: a {\em cycle cover} of a digraph is a set of vertex-disjoint directed cycles, loops counting as cycles of length one, that together contain every vertex:

\begin{theorem}[\cite{belabbas2013sparse}]\label{th:stabcond_intro}
Let $D=(V,\vec E)$ be strongly connected. Then $D$ is stable
\begin{itemize}
\item[\textup{(a)}] {\em only if} property $\mathcal N$ holds: for every $k\in\{1,\dots,n\}$ there exists $V_k\subseteq V$ with $|V_k|=k$ such that $D[V_k]$ admits a cycle cover;
\item[\textup{(b)}] {\em if} property $\mathcal S$ holds: there exists
a nested sequence of vertex sets $V_1\subset V_2\subset\cdots\subset V_n=V$ with $V_1$ consisting of a single vertex with a loop (a {\em looped} vertex), $|V_k|=k$, and each $D[V_k]$ admits a cycle cover.
\end{itemize}
\end{theorem}

The gap between the necessary $\mathcal N$ and the sufficient $\mathcal S$ is the central obstacle to a complete graph-theoretic classification of stable digraphs, and closing it remains an {\em open problem}. In the {\em symmetric} case, in which $(i,j)\in\vec E\Leftrightarrow (j,i)\in\vec E$, the gap surprisingly closes: $\mathcal N$ is both necessary {\em and} sufficient~\cite{belabbaskirkoryancdc2014decentralized}. We emphasize that a symmetric {\em SMS} does not imply the matrices in it are symmetric matrices, only that potentially non-zero entries are placed symmetrically; in fact, symmetric patterns with symmetric matrices are easily seen to be stable if and only if every vertex has a loop~\cite{belabbaskirkoryancdc2014decentralized}.

\subsection{Why random topologies?}\label{ssec:intro_random}

Structural system theory originally presumed the topology $D$ to be given. In many modern applications, however, this is not the case. A network designer may want to know not whether a specific $D$ sustains stability, but how likely a stable $D$ is among graphs drawn from a stochastic model relevant to the deployment scenario. The probability that a sampled $D$ sustains a property $\mathcal P$ measures the prevalence of that property among topologies, and is of practical value in several ways:

\begin{itemize}
\item {\em Abundance versus scarcity.} Classical structural system theory encodes a form of robustness, in that if a topology sustains a given dynamical property $\mathcal P$, this property is generic  in the case of controllability, or open in the case of stability, as we discussed earlier. Random topologies enable a stronger form  of robustness.  If $\Pr(\mathcal P)$ is close to $1$, the property is widespread and is likely to survive perturbations of the topology, such as losses of edges, additions of edges, or replacement of agents. If it is close to $0$, sustaining $\mathcal P$ requires careful design and offers little robustness margin. This is critical information for a network designer deciding whether to deploy expensive network systems in uncertain or adversarial environments~\cite{belabbas2022stable}.

\item {\em Modeling under uncertainty.} For social, biological, or estimated networks, dealing with deterministic topologies is not  meaningful: the graph must be modeled as a random object, since any estimate is affected by graph-valued noise (see, e.g.,~\cite{wolfe2013nonparametric}). The relevant question is then necessarily probabilistic.

\item {\em Inverse design.} Knowing which features of a stochastic model push $\Pr(\mathcal P)$ toward $1$ or $0$ tells the designer what to control when synthesizing a topology distribution that reliably sustains the property.
\end{itemize}

The structural stability of systems sampled from {\em symmetric} graphons has been the subject of recent work:~\cite{belabbas2022stable} treats the symmetric Erd\H{o}s--R\'enyi case;~\cite{belabbas2023geometric, belabbas2021h, belabbas2024graphons} extend the analysis to general symmetric step-graphons; and~\cite{gao2025h} addresses the boundary regime identified in~\cite{belabbas2023geometric, belabbas2021h}. In all of these works the relevant combinatorial object is the {\em edge polytope} of the skeleton, and a unified analysis is possible because, in the symmetric case, $\mathcal N$ already implies $\mathcal S$ via the resolution of~\cite{belabbaskirkoryancdc2014decentralized}.

\subsection{From symmetric to directed graphons and from edge to cycle polytopes}\label{ssec:intro_directed}

This paper extends the random topology program to the directed setting, which differs from the symmetric one in two essential ways. First, the two conditions no longer coincide: $\mathcal N$ and $\mathcal S$ must be analyzed separately, and $\mathcal S$ requires a strictly stronger combinatorial certificate, as described earlier. Second, the {\em edge polytope}, which was at the core of the earlier graphon-based results~\cite{belabbas2021h}, is  not the right object anymore. In the symmetric setting, an $n$-cycle cover (i.e., a cycle cover of all $n$ vertices, also called a cycle factor; called a Hamiltonian decomposition  in~\cite{belabbas2013sparse}) requires only that each vertex be incident to a chosen edge, and the constraint reduces to a vertex incidence count, which we showed was exactly captured by the edge polytope. In the directed setting, a cycle cover requires each vertex to have {\em outdegree and indegree exactly one}, so the constraint becomes a {\em balanced circulation}: in-flow must equal out-flow. The corresponding polytope (the {\em cycle polytope}) arises naturally from this circulation requirement.

In~\cite{chen2025hamiltonicity}, Chen studies the existence of $n$-decompositions (this corresponds to the case $k=n$ of condition $\mathcal N$ and Hamiltonian cycles for the same directed step-graphons model, and obtains conditions in terms of the node-circulation cone of the skeleton (the intersection of this cone with the hyperplane $\bfo^\top x=1$ is the cycle polytope used here) that guarantee the existence of such cycles asymptotically almost surely. 

Cycle polytopes also arise outside the present graphon problem. In polyhedral combinatorial optimization, the directed cycle polytope of a digraph is usually defined as the convex hull of the edge-incidence vectors of its simple directed cycles, and has been studied through its facets, relaxations, and connections with the asymmetric traveling-salesman polytope~\cite{balasOosten2000cycle,balasStephan2009cycleRelaxations}. Variants such as the $p$-cycle polytope appear in cardinality-constrained cycle and path problems~\cite{hartmannOzluk2001pcycle}. More recently, feasible regions of consecutive permutation-pattern densities were identified as cycle polytopes of overlap digraphs~\cite{borgaPenaguiao2020feasible}. The polytope used here is the image of the convex hull of the edge-incidence vectors of the simple directed cycles, each divided by the length of its cycle, under the linear map that sums, for each node of the skeleton, the coordinates belonging to the edges leaving that node. Taking this image discards the identity of the edges but keeps exactly the circulation constraints relevant to cycle covers.

\subsection{Main results}\label{ssec:intro_main}

Our two main theorems characterize when properties $\mathcal N$ and $\mathcal S$ hold {\em asymptotically almost surely}, in terms of structural conditions on $W$.

\xc{Skeleton, types and sampling procedure:} Before stating the conditions, we fix the notation these conditions require; the formal definitions are in Section~\ref{sec:dir_stepgraphons}. A directed step-graphon $W$ is specified by a number $q$ of {\em types}, a {\em concentration vector} $x^*$ in the standard simplex $\Delta_q\subseteq\R^q$, and {\em probabilities} $p_{ij}\in[0,1]$, $i,j\in\{1,\dots,q\}$. In a digraph $D_n\sim W$ on $n$ vertices, each vertex is assigned a type, type $i$ receiving a fraction of the vertices of $D_n$ close to $x^*_i$, and an edge from a vertex of type $i$ to a vertex of type $j$ is then present independently with probability $p_{ij}$; vertices of the same type are thus statistically interchangeable. The {\em skeleton} $S$ is the digraph  whose nodes $u_1,\dots,u_q$ are the {\em types}, with an edge $u_i\to u_j$  precisely when $p_{ij}>0$. Finally, the cycle polytope $\vec{\mathcal X}(S)$ of the previous subsection is a subset of the same simplex $\Delta_q$ in which $x^*$ lies, so the two can be compared directly. 

The conditions are:

\begin{description}
\item[(A)] {\em Loop density.} $\Lambda:=\sum_{i=1}^q x^*_i p_{ii}>0$; equivalently, a positive fraction of the vertices of $D_n$ have a loop.
\item[(B), ($\text{B}^*$)] {\em Feasibility.} $x^*\in\vec{\mathcal X}(S)$, respectively $x^*\in\rint\vec{\mathcal X}(S)$.
\item[(C)] {\em Full dimension.} $\dim\vec{\mathcal X}(S)=q-1$; that is, $\vec{\mathcal X}(S)$ is not contained in a proper affine subspace of $\Delta_q$.
\item[(D)] {\em Cycle ordering.} The nodes of $S$, $u_i$, $1\leq i\leq q$, can be ordered as $t_1,\dots,t_q$ in such a way that, for {\em every} $1 \leq r \leq q $, each node $t_1,\dots,t_r$ belongs to a  directed cycle of $S[\{t_1,\dots,t_r\}]$. 
\end{description}

The two theorems of Section~\ref{ssec:main_thms} can then be summarized as follows:
$$
\begin{array}{rclcrcl}
\text{A},\ \text{B}^*,\ \text{C}&\Longrightarrow&\Pr(\mathcal N)\to1,
&\qquad&
\text{one of A, B, C fails}&\Longrightarrow&\Pr(\mathcal N)\to0,\\[2pt]
\text{A},\ \text{B}^*,\ \text{C},\ \text{D}&\Longrightarrow&\Pr(\mathcal S)\to1,
&\qquad&
\text{one of A, B, C or D fails}&\Longrightarrow&\Pr(\mathcal S)\to0.
\end{array}
$$
The necessary and the sufficient conditions  for both $\mathcal N$ and $\mathcal S$ therefore agree except on the {\em boundary} of the cycle polytope, where sufficiency asks for $\mathrm{B}^*$ while necessity delivers only B. We call these conditions {\em essentially} necessary and sufficient. The remaining case, $x^*\in\vec{\mathcal X}(S)\setminus\rint\vec{\mathcal X}(S)$, is not an artifact of the proofs: there $\Pr(\mathcal N)$  fails to  always converge to $1$, and Proposition~\ref{prop:rint_necessary_prob1} bounds it by $1/2$. The exact asymptotics in that regime were determined for the symmetric case in~\cite{gao2025h}; for the directed case treated here they remain open. Convergence to one is nevertheless characterized {\em exactly}: since D implies C (Remark~\ref{rem:DimpliesCA}) and, when C holds, Proposition~\ref{prop:rint_necessary_prob1} rules out convergence to one on the boundary, $\Pr(\mathcal N)\to 1$ if and only if A, $\mathrm{B}^*$ and C hold, and $\Pr(\mathcal S)\to1$ if and only if $\mathrm{B}^*$ and D hold. 

Finally, Conditions A and C are in fact both implied by D (see Remark~\ref{rem:DimpliesCA} below); the formal statements below will only mention B and D for  $\mathcal S$. % \implies \mathcal N$, and because our conditions are essentially necessary and sufficient, {\em Condition D precisely encodes the gap} between necessary and sufficient conditions for structural stability.

\subsection{Proof overview}\label{ssec:intro_proofs}

The cycle polytope is what links the sampled graph to the structural conditions. The key observation, which underlies every necessity argument in the paper, is that 

\begin{quote} {\em any cycle cover on a vertex set $V'$ of $D_n$ forces the type frequency vector of $V'$ (i.e., the ratios (\# vertices of type $i$)/$|V'|$, for $i=1,\ldots,q$) into the cycle polytope}.  
\end{quote}
This holds because a cycle cover induces, after normalization, a unit circulation in the skeleton graph, that is, a nonnegative flow on each edge of $S$, of total flow one, with inflow equal to outflow at every node. Combined with the standard fact that the type frequency vector concentrates near $x^*$ asymptotically, this immediately shows that $x^*\in\vec{\mathcal X}(S)$ is necessary for both $\mathcal N$ and $\mathcal S$. This is Lemma~\ref{lem:cover_implies_Xdir}. See Figure~\ref{fig:overview} for an illustration.

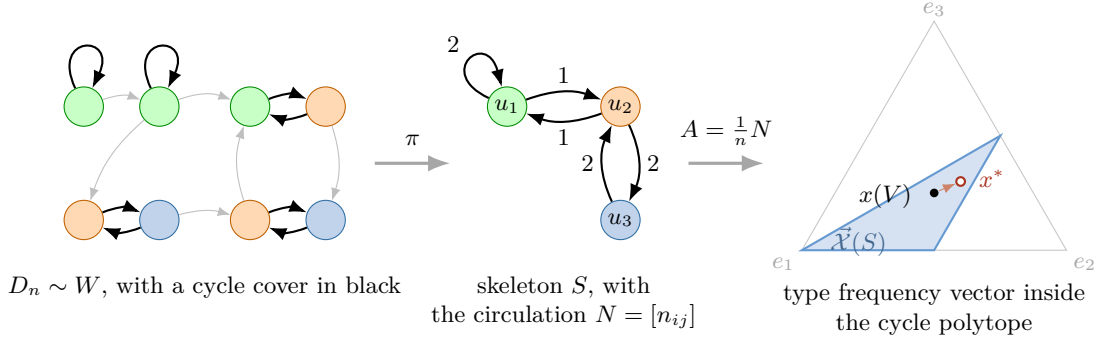
\begin{figure}[t]
\centering
\begin{tikzpicture}[
  >=Latex,
  every node/.style={font=\small},
  t1/.style={circle,draw=green!55!black,fill=green!22,minimum size=5.2mm,inner sep=0pt,font=\scriptsize},
  t2/.style={circle,draw=orange!75!black,fill=orange!30,minimum size=5.2mm,inner sep=0pt,font=\scriptsize},
  t3/.style={circle,draw=RoyalBlue!75!black,fill=RoyalBlue!22,minimum size=5.2mm,inner sep=0pt,font=\scriptsize},
  cov/.style={->,thick,black},
  oth/.style={->,gray!50,thin},
  lab/.style={font=\scriptsize}
]

\node[t1] (a1) at (0,1.15)   {};
\node[t1] (a2) at (1.0,1.15) {};
\node[t1] (a3) at (2.2,1.15) {};
\node[t2] (b1) at (3.2,1.15) {};
\node[t2] (b2) at (0,-0.35)  {};
\node[t3] (c1) at (1.0,-0.35) {};
\node[t2] (b3) at (2.2,-0.35) {};
\node[t3] (c2) at (3.2,-0.35) {};

\draw[oth] (a1) to[bend left=20] (a2);
\draw[oth] (a2) to[bend left=20] (a3);
\draw[oth] (a2) to[bend right=15] (b2);
\draw[oth] (b1) to[bend left=18] (c2);
\draw[oth] (c1) to[bend left=15] (b3);
\draw[oth] (b3) to[bend left=18] (a3);

\draw[cov] (a1) to[out=115,in=65,looseness=9] (a1);
\draw[cov] (a2) to[out=115,in=65,looseness=9] (a2);
\draw[cov] (a3) to[bend left=22] (b1);
\draw[cov] (b1) to[bend left=22] (a3);
\draw[cov] (b2) to[bend left=22] (c1);
\draw[cov] (c1) to[bend left=22] (b2);
\draw[cov] (b3) to[bend left=22] (c2);
\draw[cov] (c2) to[bend left=22] (b3);

\node[lab,anchor=north] at (1.6,-0.95) {$D_n\sim W$, with a cycle cover in black};

\draw[->,very thick,gray!70] (3.85,0.4) -- (4.85,0.4);
\node[lab,anchor=south] at (4.35,0.5) {$\pi$};

\node[t1] (u1) at (5.6,1.15) {$u_1$};
\node[t2] (u2) at (7.1,1.15) {$u_2$};
\node[t3] (u3) at (7.1,-0.35) {$u_3$};

\draw[cov] (u1) to[out=155,in=105,looseness=9] node[lab,above left=-1pt] {$2$} (u1);
\draw[cov] (u1) to[bend left=20] node[lab,above=-1pt] {$1$} (u2);
\draw[cov] (u2) to[bend left=20] node[lab,below=-1pt] {$1$} (u1);
\draw[cov] (u2) to[bend left=25] node[lab,right=-1pt] {$2$} (u3);
\draw[cov] (u3) to[bend left=25] node[lab,left=-1pt] {$2$} (u2);

\node[lab,anchor=north,align=center] at (6.35,-0.95) {skeleton $S$, with\\the circulation $N=[n_{ij}]$};

\draw[->,very thick,gray!70] (8.0,0.4) -- (9.0,0.4);
\node[lab,anchor=south,align=center] at (8.5,0.5) {$A=\tfrac1n N$};

\coordinate (e1) at (9.5,-0.75);
\coordinate (e2) at (13.0,-0.75);
\coordinate (e3) at (11.25,2.28);
\draw[gray!45] (e1)--(e2)--(e3)--cycle;
\node[lab,gray!70,below left=-2pt]  at (e1) {$e_1$};
\node[lab,gray!70,below right=-2pt] at (e2) {$e_2$};
\node[lab,gray!70,above=-2pt]       at (e3) {$e_3$};

\coordinate (v1) at (9.5,-0.75);
\coordinate (v2) at (11.25,-0.75);
\coordinate (v3) at (12.125,0.765);
\fill[RoyalBlue!14] (v1)--(v2)--(v3)--cycle;
\draw[RoyalBlue!60,thick] (v1)--(v2)--(v3)--cycle;
\node[lab,RoyalBlue!60!black,anchor=north] at (10.25,-0.30) {$\vec{\mathcal X}(S)$};

\draw[->,Mahogany!55,dashed,shorten >=2pt,shorten <=2pt] (11.25,0.0075) -- (11.60,0.159);
\node[circle,fill=black,inner sep=1.15pt] (xv) at (11.25,0.0075) {};
\node[lab,anchor=east] at (11.10,-0.03) {$x(V)$};
\node[circle,draw=Mahogany,thick,fill=white,inner sep=1.15pt] (xs) at (11.60,0.159) {};
\node[lab,Mahogany,anchor=west] at (11.70,0.20) {$x^*$};

\node[lab,anchor=north,align=center] at (11.25,-1.05) {type frequency vector inside\\the cycle polytope};
\end{tikzpicture}
\caption{The three objects used in the necessity arguments. Left: a sample $D_n$ on $n=8$ vertices, colored by type, with a cycle cover drawn in black and some further admissible edges in gray. Middle: the type map $\pi$ sends each color class to a node of $S$ and maps the cycle cover to a circulation, $n_{ij}$ counting the cover edges going from type $i$ to type $j$. Since every vertex of a cycle cover has one outgoing and one incoming edge, the row and the column sums of $N$ are both equal to the type counts. Right: dividing by $n$ turns $N$ into a unit circulation, so $x(V)=\tfrac18(3,3,2)$ lies in $\vec{\mathcal X}(S)$. As $n$ grows, $x(V)\to x^*$, whence $x^*\in\vec{\mathcal X}(S)$: this is Condition B, and it is necessary for both $\mathcal N$ and $\mathcal S$.}
\label{fig:overview}
\end{figure}

The second necessity argument is about the {\em dimension} of the polytope. When the cycle polytope has codimension one or more in the simplex, belonging to it imposes an extra linear constraint on the random type counts, in addition to  the constraint that they sum to $n$. A simple argument (see Lemma~\ref{lem:anti_conc}) shows that any  such equality constraint is satisfied  with probability vanishing as $O(n^{-1/2})$. The same approach yields the boundary bound: when $x^*$ is on the boundary of $\vec{\mathcal X}(S)$, $\Pr(\mathcal N)\leq 1/2 + o(1)$ rather than tending to $1$ (Proposition~\ref{prop:rint_necessary_prob1}).

The sufficiency arguments are constructive, and the following observation underlies sufficiency for both properties $\mathcal N$ and $\mathcal S$:  Suppose we are given a nested family of vertex sets whose type frequency vectors all lie in one fixed compact subset of $\rint\vec{\mathcal X}(S)$. On such a set the Hall inequalities of Lemma~\ref{lem:hall_Xdir} hold with a {\em constant margin}, i.e., a margin that does not decay as the sets grow. This margin is  precisely what the associated random bipartite lift of the induced subgraph requires in order to  contain a perfect matching (which corresponds to a cycle cover of $D$). With probability $1-o(1)$, the bipartite lifts of all sufficiently large members of the family admit perfect matchings at the same time, and each such matching is a cycle cover of the corresponding induced subgraph. This is Lemma~\ref{lem:dense_block_matching_combined}, and its specialization to a single nested sequence of vertex sets is Lemma~\ref{lem:uniform_chain_realization}.

For $\mathcal N$, this is essentially the entire argument. The looped vertices guaranteed by Condition A provide cycle covers of all sizes up to a fixed fraction of $n$, and ordering the sampled vertices so that the type frequency vector of every prefix stays close to that of the whole vertex set (a {\em balanced ordering}, Lemma~\ref{lem:balanced_prefix_rounding_aux}) keeps all longer prefixes near $x^*$, hence inside the compact set. Property $\mathcal S$ is harder, since the sets must form a single {\em nested sequence} starting at a looped vertex, and this is where Condition D enters. 

Geometrically, D says the following: for every $r$, the polytope $\vec{\mathcal X}(S)$ contains a point whose coordinates $t_1,\dots,t_r$ are all strictly positive and whose remaining coordinates are exactly zero. By Proposition~\ref{prop:realizable_supports}, the points of $\vec{\mathcal X}(S)$ that vanish outside $t_1,\dots,t_r$ then form a set of dimension $r-1$, the largest a set of vectors with $r$ nonzero coordinates summing to $1$ can have. Lemma~\ref{lem:integer_path} turns this into a finite list of integer type counts, each obtained from the previous one by adding a single vertex, starting from a looped node and ending at a count whose normalization lies in the relative interior of the polytope. Call a vector of type counts {\em feasible} if its normalization lies in $\vec{\mathcal X}(S)$. Since a count is feasible exactly when the corresponding blow-up of the skeleton (we call blow-up of the skeleton the digraph obtained by replacing each of its nodes  by as many vertices as the count assigns to it, and drawing an edge from every copy of $u_i$ to every copy of $u_j$ whenever $u_i\to u_j$ is an edge of $S$) can be covered by disjoint cycles (Proposition~\ref{prop:blowup_cover}), the list assembles into one finite digraph $H$, all of whose prefixes admit cycle covers. We then split $D_n$ into two independent random digraphs, or {\em layers}, whose union has the law of $D_n$, find $\Theta(n)$ disjoint copies of $H$ using the first layer, append the remaining vertices in a balanced order, and let the second layer supply the cycle covers along the whole extension.

\subsection{Outline of the paper}\label{ssec:intro_outline}

Section~\ref{sec:dir_stepgraphons} introduces directed step-graphons, the sampling procedure, the type map, cycle covers and their bipartite lifts, the properties $\mathcal N$ and $\mathcal S$, and the cycle polytope $\vec{\mathcal X}(S)$. Section~\ref{sec:conditions} states Conditions A, B, $\mathrm{B}^*$, C and D and the main theorems. Section~\ref{sec:preliminaries} collects what we need about the polytope itself: its description as a convex hull, its description by linear inequalities, the  interpretation of its integer points as cycle covers in blow-ups of the skeleton, what Condition D says about the polytope, and the balanced-ordering lemma. Section~\ref{sec:chains} builds the two devices shared by the sufficiency proofs, one deterministic and one probabilistic: a finite digraph extracted from Condition D, all of whose prefixes admit cycle covers, and a lemma that produces cycle covers along a prescribed nested sequence of vertex sets inside a random graph. Section~\ref{sec:necessary} proves the necessary conditions and the boundary bound. Sections~\ref{sec:sufficient_N} and~\ref{sec:sufficient_S} then prove sufficiency for $\mathcal N$ and for $\mathcal S$; both these sections are relatively short because, by that point, most of  the work has been done. We conclude in Section~\ref{sec:conclusion}.

\subsection*{Notation}\label{ssec:notation}
We gather here, for ease of reference, the important notation used throughout the paper. The precise definitions of some notation recalled here will be given below. For a positive integer $q$, $[q]:=\{1,\dots,q\}$ and $\bfo$ is the all-ones vector (its dimension is clear from context). We write $\aff(\cdot)$ for affine hull, and $\rint_H(F)$ for the relative interior of a  set $F$ inside an affine subspace $H\supseteq\aff(F)$. The standard simplex and its relative interior are
$$
\Delta_q:=\{x\in\R_{\geq 0}^q:\bfo^\top x=1\},
\qquad
\rint(\Delta_q):=\{x\in\R_{>0}^q:\bfo^\top x=1\}.
$$
For $x\in\R^q$ and $I\subseteq[q]$, we write
$$
x_I:=\sum_{i\in I}x_i=\mathbf 1_I^\top x,
$$ where $\mathbf 1_I\in\{0,1\}^q$ is the indicator vector of $I$. We write $A\Subset B$ when $A$ has compact closure contained in $B$. The closed $\ell_\infty$-ball of radius $r$ centered at $x$ is $\mathrm B_\infty(x,r)$, and $\dist_\infty(\cdot,\cdot)$ denotes $\ell_\infty$-distance.

$\Bin(n,p)$ and $\mathrm{Mult}(n;p)$ denote the binomial and multinomial distributions, respectively.

A directed step-graphon $W$ has $q$ types, which are recorded as the nodes $U=\{u_1,\dots,u_q\}$, block probabilities $(p_{ij})_{i,j\in[q]}$, concentration vector $x^*\in\Delta_q$, and skeleton digraph $S=(U,F)$ with $(u_i\to u_j)\in F\Leftrightarrow p_{ij}>0$. We identify a node $u_i$ with its index $i$ when writing coordinates and subsets of $U$. The set of looped nodes of $S$ is
$$
L:=\{i\in[q]:p_{ii}>0\}.
$$
Two derived quantities recur throughout: the loop density and the minimum edge probability:
$$
\Lambda:=\sum_{i=1}^q x_i^*p_{ii},
\qquad
p_{\min}:=\min\{p_{ij}:(u_i\to u_j)\in F\}\in(0,1].
$$
For $I\subseteq U$, the out-neighbor and in-neighbor sets in $S$ are
$$
N^+(I):=\{u_j:\exists\,u_i\in I,\,(u_i\to u_j)\in F\},
\quad
N^-(I):=\{u_j:\exists\,u_i\in I,\,(u_j\to u_i)\in F\}.
$$
For nonempty proper $I\subsetneq U$, the {\em Hall functional} is the linear map
$\Phi_I(x):=x_{N^+(I)}-x_I$; Lemma~\ref{lem:hall_Xdir} shows that $\vec{\mathcal X}(S)$ is cut out of $\Delta_q$ by the inequalities $\Phi_I\ge0$, and several arguments require a uniform {\em margin} $\Phi_I\ge\delta>0$.

Vertices of $D_n\sim W$ inherit a type via the type map $\pi:V\to U$ (Definition~\ref{def:type_map}). For $V'\subseteq V$,
$$
N_i(V'):=|\{v\in V':\pi(v)=u_i\}|,
\qquad
x(V'):=\frac{1}{|V'|}\bigl(N_1(V'),\dots,N_q(V')\bigr)
$$ are the {\em type counts} and the {\em type frequency vector} of $V'$. We write $N_i:=N_i(V)$ and $x:=x(V)$ for the global versions, and $B(D)$ for the bipartite lift of a directed graph $D$ (Definition~\ref{def:bipartite_rep}).

The set of unit circulations on $S$ is
$$
\mathcal A(S):=\bigl\{A\in\R_{\ge0}^{q\times q}:\ A\bfo=A^\top\bfo,\ \bfo^\top A\bfo=1,\ a_{ij}=0\text{ if }(u_i\to u_j)\notin F\bigr\},
$$ and the cycle polytope is
$$
\vec{\mathcal X}(S):=\{A\bfo:A\in\mathcal A(S)\}\subseteq\Delta_q.
$$
For $T\subseteq U$ nonempty, the affine subspace $H_T$ and the support face $F_T$ are
$$
H_T:=\{x\in\R^q:\bfo^\top x=1,\ x_i=0\text{ for }i\notin T\},
\qquad
F_T:=\vec{\mathcal X}(S)\cap H_T.
$$

\section{Setup and definitions}\label{sec:dir_stepgraphons}

This section gathers the definitions used throughout the paper.
\subsection{Directed step-graphons}\label{ssec:dir_stepgraphons}

\begin{definition}[Directed step-graphon and concentration vector]\label{def:dir_stepgraphon}
A {\em directed step-graphon} is a measurable function $W:[0,1]^2\to[0,1]$ for which there exists an increasing sequence $0=\sigma_0<\sigma_1<\cdots< \sigma_q=1$ such that $W$ is constant on each rectangle $[\sigma_{i-1},\sigma_i)\times[\sigma_{j-1},\sigma_j)$ for all $1\leq i,j\leq q$. We denote the corresponding constant by $p_{ij}\in[0,1]$ and define the {\em concentration vector} $x^*=(x_1^*,\dots,x_q^*)$ by $x_i^*:=\sigma_i-\sigma_{i-1}$.
\end{definition}

Since the sequence $\sigma_0<\cdots<\sigma_q$ is strictly increasing, every $x_i^*=\sigma_i-\sigma_{i-1}>0$; thus $x^*$ has full support. If $p_{ij}=p_{ji}$ for all pairs $1\leq i,j\leq q$, then $W$ is a (symmetric) step-graphon.

%\subsection{Sampling}\label{ssec:sampling}

We sample directed graphs from a directed step-graphon model as in the symmetric case, except that the two orientations of each pair of vertices, and each loop, are sampled independently.

\begin{definition}[Sampling procedure]\label{def:sampling}
Let $W$ be a directed step-graphon. A random directed graph $D_n=(V,\vec E)$ sampled from $W$, denoted $D_n\sim W$, is obtained as follows: first sample $y_1,\dots,y_n\sim\mathrm{Unif}[0,1]$ independently; then for each ordered pair $(a,b)\in\{1,\dots,n\}^2$, include the directed edge $v_a\to v_b$ independently with probability $W(y_a,y_b)$.
\end{definition}

\subsection{Skeleton digraph and type map}\label{ssec:types_ecv}

The skeleton digraph of $W$ records which blocks may produce edges, i.e., which blocks support a non-zero value for $W$:

\begin{definition}[Directed skeleton digraph]\label{def:dir_skeleton}
The {\em directed skeleton} is the directed graph $S=(U,F)$ on $U=\{u_1,\dots,u_q\}$ with $(u_i\to u_j)\in F$ if and only if $p_{ij}>0$. The node $u_i$ records the $i$th type of $W$.
\end{definition}

The {\em type of a vertex} $v_a$ of $D_n$ is determined by the location of its uniform coordinate $y_a$:

\begin{definition}[Type map]\label{def:type_map}
Let $D_n\sim W$ and let $y_a$ be the coordinate of vertex $v_a$ as in Definition~\ref{def:sampling}. Define the type map $\pi:V\to U$ by $\pi(v_a)=u_i$ if $y_a\in[\sigma_{i-1},\sigma_i)$.
\end{definition}

The type map is a digraph homomorphism from $D_n$ to the skeleton: for every edge $v_a\to v_b$ in $D_n$, we have $(\pi(v_a)\to\pi(v_b))\in F$. We say $v$ has type $i$ if $\pi(v)=u_i$.

We call an ordered pair of nodes $(u_i,u_j)$, and any edge of a sampled graph $D$ from a type-$i$ vertex to a type-$j$ vertex, {\em admissible} if $(u_i\to u_j)\in F$.

\subsection{Type frequencies}\label{ssec:emp_concentration}

We denote by $N_i:=|\{v_a\in V:\pi(v_a)=u_i\}|$ the number of vertices in $D_n$ of type $i$, and define the {\em type frequency vector} of $D_n$ to be $x:=(N_1,\dots,N_q)/n$. Thus $x$ is the sample or empirical  counterpart  of the concentration vector $x^*$, and Lemma~\ref{lem:multinomial_conc} shows that $x$ concentrates near $x^*$. We extend this to induced subgraphs: for $V'\subseteq V$, $N_i(V'):=|\{v\in V':\pi(v)=u_i\}|$ and 
$$x(V'):=(N_1(V'),\dots,N_q(V'))/|V'|\in\R_{\geq 0}^q.$$

\subsection{Cycle covers and bipartite lifts}\label{ssec:cyclecover_def}

\begin{definition}[Cycle cover]\label{def:cycle_cover}
Let $D=(V,\vec E)$ be a directed graph (allowing loops). For $V'\subseteq V$ we say $D[V']$ {\em admits a cycle cover} if there exists $\vec F\subseteq \vec E(D[V'])$ such that in $(V',\vec F)$ every vertex has outdegree $1$ and indegree $1$.
\end{definition}

Equivalently, $V'$ is partitioned into the vertex sets of vertex-disjoint directed cycles of $D[V']$, loops counting as cycles of length one. We call a cycle cover of an induced subdigraph on $k$ vertices a {\em $k$-cycle cover}; thus an $n$-cycle cover of $D_n$ is a cycle cover of $D_n$ itself (a {\em Hamiltonian decomposition} in the terminology of~\cite{belabbas2013sparse}), and a $1$-cycle cover is a loop. A vertex with a loop is called {\em looped}.

\begin{definition}[Bipartite lift]\label{def:bipartite_rep}
Given a directed graph $D=(V,\vec E)$, define a bipartite graph $B(D)$ with left part $L=\{v^{\mathrm{out}}:v\in V\}$ and right part $R=\{v^{\mathrm{in}}:v\in V\}$, and edges $v^{\mathrm{out}}w^{\mathrm{in}}\in E(B(D))$ if and only if $v\to w\in\vec E$.
\end{definition}

The letters $L$ (with a slight abuse of notation, since $L$ also denotes the set of looped nodes of $S$; the context will make the distinction clear.)  and $R$ refer to the {\em left} and {\em right} sides of $B(D)$.

We say that a subset $X\subseteq L$ or $X\subseteq R$ is a {\em Hall obstruction}, or simply an {\em obstruction}, if $|N(X)|<|X|$. The terminology is derived from Hall's theorem, which states that a perfect matching in the bipartite graph $B(D)$ exists if and only if $B(D)$ contains no obstruction $X\subseteq L$, i.e., $|N(X)|\ge|X|$ for every $X\subseteq L$ (recall that $|L|=|R|$).

\subsection{Properties $\mathcal N$ and $\mathcal S$}\label{ssec:NS_def}

The structural-stability background was sketched in Section~\ref{ssec:intro_struct}: for a strongly connected digraph $D=(V,\vec E)$ on $n$ vertices, stability of the corresponding sparse matrix space is governed by two graph-theoretic conditions of~\cite{belabbas2013sparse}: the necessary cycle-cover condition and the sufficient nested-sequence condition. We record them here precisely, since they are the subject of the rest of the paper:

\begin{definition}[Properties $\mathcal N$ and $\mathcal S$]\label{def:NS}
Let $D=(V,\vec E)$ be a digraph on $n$ vertices. We say that $D$ has {\em property $\mathcal N$} if for each $k\in\{1,\dots,n\}$, there exists a $k$-vertex induced subdigraph of $D$ that admits a cycle cover. We say that $D$ has {\em property $\mathcal S$} if there exists a nested sequence $J_1\subset J_2\subset\cdots\subset J_n=V$ with $|J_k|=k$ such that each $D[J_k]$ admits a cycle cover and $J_1$ is a single looped vertex.
\end{definition}

For each fixed digraph, $\mathcal S\Rightarrow\mathcal N$. In the symmetric case, $\mathcal N$ and $\mathcal S$ have been shown to be equivalent~\cite{belabbaskirkoryancdc2014decentralized}.

\subsection{The cycle polytope}\label{ssec:cyclepoly_def}

To translate the cycle-cover requirements of $\mathcal N$ and $\mathcal S$ into testable conditions on the step-graphon $W$, we associate to the skeleton $S$ a polytope of attainable type frequency vectors.

\begin{definition}[Unit circulations]\label{def:A_S}
Let $S=(U,F)$ be a directed graph on $q$ nodes. A {\em unit circulation} on $S$ is a $q\times q$ matrix $A=[a_{ij}]$ with $A\geq 0$ entrywise, $a_{ij}=0$ whenever $(u_i\to u_j)\notin F$, and $A\bfo=A^\top\bfo$ with $\bfo^\top A\bfo=1$. We write $\mathcal A(S)$ for the set of unit circulations on $S$.
\end{definition}

\begin{definition}[Cycle polytope]\label{def:cycle_polytope}
Let $S=(U,F)$ be a directed graph on $q$ nodes. Define
$$
\vec{\mathcal X}(S):=\{x\in\R^q:\ x=A\bfo\ \text{for some }A\in\mathcal A(S)\}.
$$
\end{definition}

The entry $a_{ij}$ is the {\em flow} on the edge $u_i\to u_j$. It is nonnegative and vanishes if $u_i\to u_j$ is not an edge of $S$, the total flow $\bfo^\top A\bfo$ equals $1$ (hence {\em unit}) and $A\bfo=A^\top\bfo$ is flow conservation (inflow equals outflow) at each node. The vector $A\bfo$ lists the total flow out of each node, which flow conservation makes equal to the total flow into it; since the total flow is $1$, it lies in $\Delta_q$. We call it the {\em type frequency vector} of $A$, a name justified by Proposition~\ref{prop:blowup_cover}: these are exactly the type frequency vectors of vertex sets that admit a cycle cover.

The cycle polytope admits an equivalent description as the convex hull of the type frequency vectors of directed cycles in $S$, which we now define for later use. For a cycle $C=(i_1\to i_2\to\cdots\to i_\ell\to i_1)$ in $S$ of length $\ell$ (not necessarily simple, i.e., a closed directed walk, with the convention $i_{\ell+1}:=i_1$), define its {\em cycle matrix} $A^{C}\in\R^{q\times q}$ and its {\em type frequency vector} $x^C\in\Delta_q$ as 
$$
A^{C}_{ij}:=\frac{1}{\ell}\,\#\{r\in[\ell]:(u_{i_r}\to u_{i_{r+1}})=(u_i\to u_j)\}\quad \mbox{ and }\quad  x^{C}:=A^{C}\bfo,
$$ 
respectively. Then
$$
\vec{\mathcal X}(S) = \conv \{ x^C \mid C \text{ is a directed cycle in } S\},
$$
and simple cycles already suffice, so $\vec{\mathcal X}(S)$ is the convex hull of finitely many points. This is Lemma~\ref{lem:cycle_polytope_conv_cycles}. 
Additional properties of $\vec{\mathcal X}(S)$ are developed in Section~\ref{sec:preliminaries}; Figure~\ref{fig:cycle_polytope_example} shows a worked example.

\begin{figure}[t]
\centering
\begin{tikzpicture}[
>=Latex,
every node/.style={font=\small},
tp/.style={circle, draw, minimum size=7mm, inner sep=0pt},
looped/.style={circle, draw, fill=green!12, minimum size=7mm, inner sep=0pt}
]
% ---- left: skeleton S ----
\begin{scope}[yshift=0.9cm]
\node[looped] (s1) at (0,0) {$1$};
\node[tp] (s2) at (1.6,0) {$2$};
\node[tp] (s3) at (3.2,0) {$3$};
\draw[->] (s1) edge[loop left] ();
\draw[->] (s1) to[bend left=18] (s2);
\draw[->] (s2) to[bend left=18] (s1);
\draw[->] (s2) to[bend left=18] (s3);
\draw[->] (s3) to[bend left=18] (s2);
\node[font=\footnotesize, align=center] at (1.6,-1.2)
{skeleton $S$: cycles\\ loop $1$,\ $1\!\leftrightarrow\!2$,\ $2\!\leftrightarrow\!3$};
\end{scope}
% ---- right: simplex + cycle polytope ----
\begin{scope}[xshift=6cm]
\coordinate (e1) at (0,0);
\coordinate (e2) at (4,0);
\coordinate (e3) at (2,3.46);
\fill[gray!8] (e1)--(e2)--(e3)--cycle;
\draw[gray!45] (e1)--(e2)--(e3)--cycle;
\node[below left] at (e1) {\footnotesize $e_1$};
\node[below right] at (e2) {\footnotesize $e_2$};
\node[above] at (e3) {\footnotesize $e_3$};
\coordinate (v1) at (0,0);
\coordinate (v2) at (2,0);
\coordinate (v3) at (3,1.73);
\fill[blue!16] (v1)--(v2)--(v3)--cycle;
\draw[blue!55,thick] (v1)--(v2)--(v3)--cycle;
\node[circle,fill=black,inner sep=1.2pt] at (v1) {};
\node[circle,fill=black,inner sep=1.2pt] at (v2) {};
\node[circle,fill=black,inner sep=1.2pt] at (v3) {};
\node[font=\scriptsize, below=1pt] at (v2) {$(\tfrac12,\tfrac12,0)$};
\node[font=\scriptsize, right=1pt] at (v3) {$(0,\tfrac12,\tfrac12)$};
\node[font=\scriptsize, above right] at (v1) {$(1,0,0)$};
\node[circle,fill=blue!70!black,inner sep=1.6pt,label={[font=\scriptsize]above:{$x^*$: $B^*$}}] at (1.667,0.577) {};
\node[circle,fill=red!75!black,inner sep=1.6pt,label={[font=\scriptsize]right:{$x^*$: $B$ only}}] at (2.5,0.865) {};
\end{scope}

% ---- step-graphons producing the two marked concentration vectors ----
\begin{scope}[yshift=-3.7cm, xshift=2.4cm]
\begin{scope}
\fill[Gray!70!black!55] (0,0.9)   rectangle (0.9,1.8);
\fill[Gray!40!black!30] (0.9,0.9) rectangle (1.5,1.8);
\fill[Gray!60!black!60] (0,0.3)   rectangle (0.9,0.9);
\fill[Gray!50!black!40] (1.5,0.3) rectangle (1.8,0.9);
\fill[Gray!70!black!75] (0.9,0)   rectangle (1.5,0.3);
\draw[gray!55,very thin] (0.9,0)--(0.9,1.8) (1.5,0)--(1.5,1.8) (0,0.9)--(1.8,0.9) (0,0.3)--(1.8,0.3);
\draw[black,thick] (0,0) rectangle (1.8,1.8);
\node[font=\tiny,above=0pt] at (0.45,1.8) {$1$};
\node[font=\tiny,above=0pt] at (1.2,1.8)  {$2$};
\node[font=\tiny,above=0pt] at (1.65,1.8) {$3$};
\node[font=\scriptsize,below=1pt,align=center] at (0.9,0) {$x^*=(\tfrac12,\tfrac13,\tfrac16)$\\relative interior: $\mathrm{B}^*$};
\end{scope}
\begin{scope}[xshift=3.4cm]
\fill[Gray!70!black!55] (0,1.35)    rectangle (0.45,1.8);
\fill[Gray!40!black!30] (0.45,1.35) rectangle (1.35,1.8);
\fill[Gray!60!black!60] (0,0.45)    rectangle (0.45,1.35);
\fill[Gray!50!black!40] (1.35,0.45) rectangle (1.8,1.35);
\fill[Gray!70!black!75] (0.45,0)    rectangle (1.35,0.45);
\draw[gray!55,very thin] (0.45,0)--(0.45,1.8) (1.35,0)--(1.35,1.8) (0,1.35)--(1.8,1.35) (0,0.45)--(1.8,0.45);
\draw[black,thick] (0,0) rectangle (1.8,1.8);
\node[font=\tiny,above=0pt] at (0.225,1.8) {$1$};
\node[font=\tiny,above=0pt] at (0.9,1.8)   {$2$};
\node[font=\tiny,above=0pt] at (1.575,1.8) {$3$};
\node[font=\scriptsize,below=1pt,align=center] at (0.9,0) {$x^*=(\tfrac14,\tfrac12,\tfrac14)$\\boundary: B only};
\end{scope}
\end{scope}
\end{tikzpicture}
\caption{The cycle polytope of a graph. Left: a strongly connected graph $S$ on $q=3$ nodes (node $1$ looped, in green). Its simple directed cycles are the loop at $1$, the $2$-cycle $1\!\leftrightarrow\!2$, and the $2$-cycle $2\!\leftrightarrow\!3$, with $x^{C}=(1,0,0)$, $(\tfrac12,\tfrac12,0)$, and $(0,\tfrac12,\tfrac12)$. Right: the simplex $\Delta_3$ (gray) and the cycle polytope $\vec{\mathcal X}(S)=\conv\{x^{C}\}$ (blue), here $2$-dimensional, so Condition C holds (the conditions are stated in Section~\ref{ssec:conditions}). A concentration vector $x^*$ in the relative interior satisfies $\mathrm{B}^*$ (hence B); one on the polytope boundary but interior to the simplex satisfies  B but not $\mathrm{B}^*$. Bottom: two step-graphons with skeleton $S$ whose concentration vectors are the two marked points, drawn on $[0,1]^2$ with the block $(i,j)$ shaded according to $p_{ij}$ and white where $p_{ij}=0$. The widths of the three strips are the entries of $x^*$; the pattern of shaded blocks is the same in both, since it records only the skeleton $S$, and it is the strip widths  that move $x^*$ from the relative interior to the boundary.
}
\label{fig:cycle_polytope_example}
\end{figure}
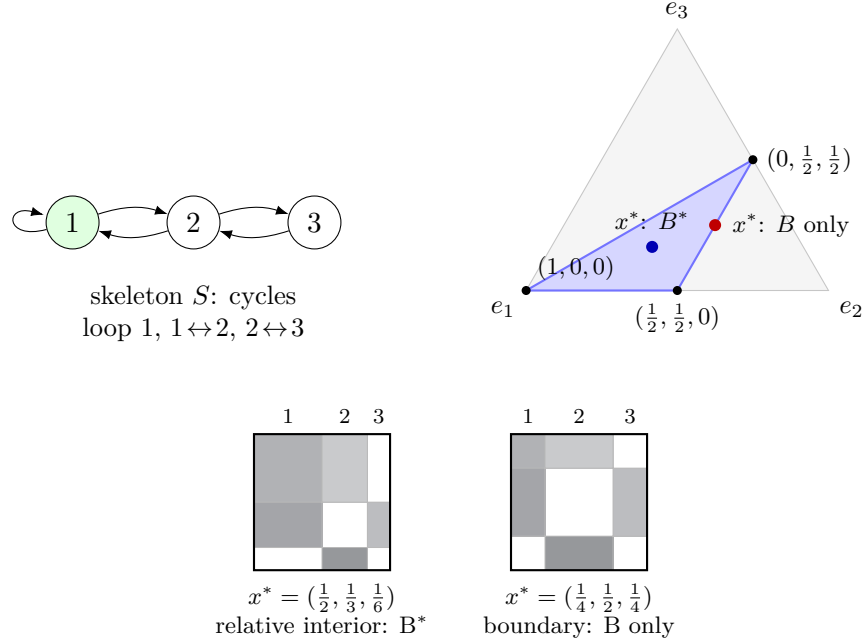

\section{Conditions and main results}\label{sec:conditions}

We now state the structural conditions on $W$ and the main theorems they yield. From this point onward, we assume that the skeleton $S=(U,F)$ is strongly connected.

\subsection{Conditions A--D}\label{ssec:conditions}

\begin{description}
\item[(A)] Positive loop density: $\Lambda:=\sum_{i=1}^q x_i^*p_{ii}>0$.
\item[(B)] Feasibility: $x^*\in\vec{\mathcal X}(S)$.
\item[($\mathbf{B}^*$)] Strict feasibility: $x^*\in\rint\vec{\mathcal X}(S)$.
\item[(C)] Full dimension: $\dim\vec{\mathcal X}(S)=q-1$.
\item[(D)] Cycle ordering: there exists an ordering $t_1,\dots,t_q$ of the nodes of $U$, that is, a relabelling of $u_1,\dots,u_q$, such that, writing $T_r:=\{t_1,\dots,t_r\}$, every node of $S[T_r]$ lies on a directed cycle of $S[T_r]$, for every $r=1,\dots,q$. 
\end{description}

\begin{remark}\label{rem:DimpliesCA}
Condition D implies both A and C. The case $r=1$ gives a loop at $t_1$, and every coordinate of $x^*$ is positive by Definition~\ref{def:dir_stepgraphon}, so $\Lambda\geq x^*_{t_1}p_{t_1t_1}>0$, which is A. Taking $r=q$ in Proposition~\ref{prop:realizable_supports} gives $\rint_{H_U}(F_U)\neq\varnothing$, that is, $\dim\vec{\mathcal X}(S)=q-1$, which is C.
\end{remark}

Conditions B and $\mathrm{B}^*$ play different roles, as was already observed in~\cite{belabbas2021h}: B is necessary for $\Pr(\mathcal N)\to 1$, but the matching sufficient condition is the strict-feasibility version $\mathrm{B}^*$. The boundary gap between the two, i.e., the case $x^*\in\vec{\mathcal X}(S)\setminus\rint\vec{\mathcal X}(S)$, is where the asymptotic probability of $\mathcal N$ is bounded strictly below $1$, in fact at most $1/2$ in the $\limsup$ (Proposition~\ref{prop:rint_necessary_prob1}). Figure~\ref{fig:cycle_polytope_example} (right panel) illustrates the distinction. 
%A set $T$ of types is the support of a point of $\vec{\mathcal X}(S)$ exactly when every node of $S[T]$ lies on a directed cycle of $S[T]$ (Proposition~\ref{prop:realizable_supports}); Condition D therefore asks for an ordering in which each new type lies on a cycle through the types already taken. 
Finally, see Figure~\ref{fig:Dadm_examples} for an illustration of Condition D.

Finally, we note that Condition D can be checked greedily. Indeed,  start from a looped node and repeatedly adding any node that lies on a directed cycle of the subgraph of $S$ induced by that node and the nodes already chosen. A node that can be added remains so when additional nodes are chosen, since the cycle through still persists. Consequently, if some ordering $t_1,\dots,t_q$ satisfies D, the procedure outlined above cannot stop before reaching all of $U$. Thus D holds if and only if the greedy procedure reaches all of $U$.

\begin{figure}[t]
\centering
\begin{tikzpicture}[
>=Latex,
node distance=2.2cm,
every node/.style={font=\small},
type/.style={circle, draw, minimum size=7mm, inner sep=0pt},
looped/.style={type, fill=green!12},
note/.style={align=left, font=\small}
]

\begin{scope}[xshift=0cm]
\node[looped] (a1) at (0,0) {$1$};
\node[type] (a2) at (2.2,0) {$2$};
\node[type] (a3) at (1.1,1.9) {$3$};

\draw[->, thick] (a1) edge[loop below] ();
\draw[->, thick] (a1) edge[bend left=10] (a2);
\draw[->, thick] (a2) edge[bend left=10] (a1);
\draw[->, thick] (a2) -- (a3);
\draw[->, thick] (a3) -- (a1);

\node[note, anchor=north] at (1.1,-1.0)
{\textbf{Satisfies $D$}\\
Ordering: $t_1=1,\ t_2=2,\ t_3=3$\\[1mm]
Cycles certifying each prefix:\\
$T_1=\{1\}$: loop at $1$\\
$T_2=\{1,2\}$: $1\leftrightarrow 2$\\
$T_3=\{1,2,3\}$: $1\to2\to3\to1$};
\end{scope}

\begin{scope}[xshift=8.2cm]
\node[looped] (b1) at (0,0) {$1$};
\node[type] (b2) at (2.2,0) {$2$};
\node[type] (b3) at (1.1,1.9) {$3$};

\draw[->, thick] (b1) edge[loop below] ();
\draw[->, thick] (b1) -- (b2);
\draw[->, thick] (b2) -- (b3);
\draw[->, thick] (b3) -- (b1);

\node[note, anchor=north] at (1.1,-1.0)
{\textbf{Fails $D$}\\
Strongly connected, but $L=\{1\}$, so $t_1=1$ is forced.\\[1mm]
No two-node prefix works:\\
$\{1,2\}$ fails since $2$ has no outgoing edge inside $\{1,2\}$;\\
$\{1,3\}$ fails since $3$ has no incoming edge inside $\{1,3\}$.\\
Thus no ordering $t_1,t_2,t_3$ can satisfy $D$.};
\end{scope}

\end{tikzpicture}
\caption{Illustration of the cycle-ordering Condition D for strongly connected skeleton digraphs.}
\label{fig:Dadm_examples}
\end{figure}

\subsection{Main theorems}\label{ssec:main_thms}

\begin{theorem}[Property $\mathcal N$]\label{thm:maincondN}
Let $W$ be a directed step-graphon with strongly connected skeleton $S$ and concentration vector $x^*$. Let $D_n\sim W$. If any of Conditions A, B, C is {\em not satisfied}, then $\lim_{n\to\infty}\Pr(\mathcal N)=0$. Conversely, if Conditions A, $\mathrm{B}^*$ and C are satisfied, then there exist constants $C_W,c_W>0$, depending only on $W$, such that $\Pr(\mathcal N)\ge1-C_We^{-c_Wn}$ for all $n$; in particular, $\lim_{n\to\infty}\Pr(\mathcal N)=1$.
\end{theorem}

\begin{theorem}[Property $\mathcal S$]\label{thm:maincondS}
Let $W$ be a directed step-graphon with strongly connected skeleton $S$ and concentration vector $x^*$. Let $D_n\sim W$. If Condition B or Condition D fails, then $\lim_{n\to\infty}\Pr(\mathcal S)=0$. Conversely, if Conditions $\mathrm{B}^*$ and D hold, then there exist constants $C_W,c_W>0$, depending only on $W$, such that $\Pr(\mathcal S)\ge1-C_We^{-c_Wn}$ for all $n$; in particular, $\lim_{n\to\infty}\Pr(\mathcal S)=1$.
\end{theorem}
%Condition D implies both A and C (Remark~\ref{rem:DimpliesCA}), so the criterion for $\mathcal S$ is feasibility together with the support chain, and listing A and C alongside them would be redundant; in particular the necessity half above is equivalent to saying that $\Pr(\mathcal S)\to0$ as soon as any one of A, B, C, D fails. 

Comparing the two theorems, the  difference between $\mathcal N$ and $\mathcal S$ is that $\mathcal N$ asks only, besides loops (Condition A), that the polytope contain $x^*$ and have full dimension, while $\mathcal S$ asks in addition that its support be reachable one node at a time, starting from a looped node and passing only through supports of points of the polytope (Condition D).

\begin{example}[$\mathcal N$ and $\mathcal S$ separate on directed step-graphons]\label{ex:NS_separation}
Let $q=4$ and let $S$ have the edges: a loop at $1$, $1\to2$, $4\to1$, and $2\leftrightarrow3$, $3\leftrightarrow4$, $2\leftrightarrow4$; thus $L=\{1\}$ and $S$ is strongly connected. The type frequency vectors $e_1$, $(0,\tfrac12,\tfrac12,0)$, $(0,0,\tfrac12,\tfrac12)$, $(0,\tfrac12,0,\tfrac12)$ of four of its simple cycles, namely the loop at $1$ and the $2$-cycles $2\leftrightarrow3$, $3\leftrightarrow4$ and $2\leftrightarrow4$, are affinely independent, so Condition C holds, and their average $(\tfrac14,\tfrac14,\tfrac14,\tfrac14)$ lies in $\rint\vec{\mathcal X}(S)$. Take any block probabilities supported on $S$ and $x^*=(\tfrac14,\tfrac14,\tfrac14,\tfrac14)$: Conditions A, $\mathrm{B}^*$, C hold. But Condition D fails: $t_1=1$ is forced, and no two-node prefix $\{1,i\}$ is the support of a point of $\vec{\mathcal X}(S)$, since none of $S[\{1,2\}]$, $S[\{1,3\}]$, $S[\{1,4\}]$ contains a cycle through $i$. By Theorems~\ref{thm:maincondN} and~\ref{thm:maincondS}, $\Pr(\mathcal N)\to1$ while $\Pr(\mathcal S)\to 0$.
\end{example}

\section{Preliminary results}\label{sec:preliminaries}

This section collects what we need to know about the cycle polytope, together with two facts about the sampled graph that are used throughout. Section~\ref{ssec:concentration_lemmas} records the concentration of the type frequency vector and of the loop count. Section~\ref{ssec:cycle_cover_lemmas} recalls Hall's theorem and the equivalence between cycle covers and perfect matchings of the bipartite lift. Section~\ref{ssec:cycle_polytope_lemmas} describes $\vec{\mathcal X}(S)$ in the two ways discussed above: as a convex hull of the type frequency vectors of cycles and as a system of linear inequalities, and identifies its integer points with cycle covers of blow-ups of the skeleton. Section~\ref{ssec:Dgeometric} elaborates the graph-theoretic and the geometric descriptions of Condition D. Section~\ref{ssec:steinitz_balancing} proves the balanced-ordering lemma, a {\em Steinitz} rearrangement that both sufficiency proofs use to add vertices.

\subsection{Concentration of the type frequency vector}\label{ssec:concentration_lemmas}

The two basic distributional facts about $D_n\sim W$ are that the type counts are multinomial, so that the type frequency vector concentrates near $x^*$, and that the number of looped vertices is binomial.

\begin{lemma}\label{lem:multinomial_conc}
Let $D_n\sim W$ and let $x$ be its type frequency vector. Then $n x=(N_1,\dots,N_q)\sim\mathrm{Mult}(n;x^*)$. Moreover, for every $\varepsilon>0$ there exists $c=c(\varepsilon,x^*)>0$ such that
$$
\Pr\bigl(\|x-x^*\|_2>\varepsilon\bigr)\ \le\ 2\blue{q}\,e^{-c n}.
$$
\end{lemma}

\begin{proof}
Since the $y_a$ are i.i.d.\ with $\Pr(y_a\in[\sigma_{i-1},\sigma_i))=x_i^*$, the count vector $(N_1,\dots,N_q)$ is $\mathrm{Mult}(n;x^*)$. Each marginal $N_i\sim\Bin(n,x_i^*)$ obeys Hoeffding's bound $\Pr(|N_i-nx_i^*|>\delta n)\leq 2e^{-c_i n}$ with $c_i=c_i(\delta,x_i^*)>0$. A union bound over $i$ together with $\|x-x^*\|_2\le\sqrt q\,\|x-x^*\|_\infty$ gives the claim with $\delta:=\varepsilon/\sqrt q$ and $c:=\min_i c_i$.
\end{proof}

\begin{lemma}[Loop count]\label{lem:binomial_loops}
Let $D_n\sim W$ where $W$ is a directed step-graphon with concentration vector $x^*$ and block probabilities $p_{ij}$. Recall $\Lambda:=\sum_{i=1}^q x_i^*\,p_{ii}$. Let $L_n$ be the number of looped vertices in $D_n$. Then $L_n\sim\Bin(n,\Lambda)$.
\end{lemma}

\begin{proof}
For each vertex $v_a$, the loop indicator $\mathbf 1\{v_a\to v_a\}$ is Bernoulli with parameter
$$
\Pr(v_a\to v_a)=\E[W(y_a,y_a)]
=\sum_{i=1}^q \Pr(y_a\in[\sigma_{i-1},\sigma_i))\,p_{ii}
=\sum_{i=1}^q x_i^*p_{ii}
=\Lambda.
$$
Moreover, the pairs $(y_a,\mathbf 1\{v_a\to v_a\})$, $a\in[n]$, are i.i.d.\ (the coordinates are i.i.d.\ and loops are sampled independently across vertices by Definition~\ref{def:sampling}), so the loop indicators are i.i.d.\ Bernoulli($\Lambda$) and $L_n\sim\Bin(n,\Lambda)$.
\end{proof}

\subsection{Cycle covers as bipartite matchings}\label{ssec:cycle_cover_lemmas}

It is well known that  cycle covers in $D$ are in one-to-one correspondence with perfect matchings of the bipartite lift; see Figure~\ref{fig:cyclecover_matching} for an illustration:

\begin{lemma}[Cycle cover $\equiv$ perfect matching]\label{lem:cyclecover_matching}
Let $D=(V,\vec E)$ be a directed graph and $V'\subseteq V$. Then $D[V']$ admits a cycle cover if and only if $B(D[V'])$ admits a perfect matching.
\end{lemma}
The result is standard; we include a short proof for completeness.
\begin{proof}
A cycle cover of $D[V']$ is the same as a permutation $\varphi:V'\to V'$ such that $v\to\varphi(v)$ is an edge of $D[V']$ for every $v\in V'$. The corresponding edges $v^{\mathrm{out}}\varphi(v)^{\mathrm{in}}$ form a perfect matching of $B(D[V'])$. Conversely, a perfect matching of $B(D[V'])$ matches each left copy $v^{\mathrm{out}}$ to a unique right copy $\varphi(v)^{\mathrm{in}}$ and matches every right copy exactly once; thus $\varphi$ is a permutation of $V'$, and the directed edges $v\to\varphi(v)$ form a cycle cover of $D[V']$.
\end{proof}

\begin{figure}[t]
\centering
\begin{tikzpicture}[
>=Latex,
every node/.style={font=\small},
dv/.style={circle, draw, fill=black!5, minimum size=6mm, inner sep=0pt},
cov/.style={very thick, blue!70!black},
extra/.style={gray!55}
]
% ---- left: digraph D with a cycle cover ----
\begin{scope}
\node[dv] (1) at (0,1.7) {$1$};
\node[dv] (2) at (2,1.7) {$2$};
\node[dv] (3) at (0,-0.3) {$3$};
\node[dv] (4) at (1,-1.7) {$4$};
\node[dv] (5) at (2,-0.3) {$5$};
\draw[->,cov] (1) to[bend left=15] (2);
\draw[->,cov] (2) to[bend left=15] (1);
\draw[->,cov] (3) -- (4);
\draw[->,cov] (4) -- (5);
\draw[->,cov] (5) -- (3);
\draw[->,extra] (1) -- (5);
\draw[->,extra] (3) to[bend left=12] (2);
\node[font=\footnotesize] at (1,-2.6) {digraph $D$; cover in blue};
\end{scope}
% ---- right: bipartite lift B(D) with perfect matching ----
\begin{scope}[xshift=6.6cm]
\foreach \i in {1,...,5} {
\node[circle,fill=black,inner sep=1.5pt,label=left:{\footnotesize $\i^{\mathrm{out}}$}] (L\i) at (0, 2.5-0.9*\i) {};
\node[circle,fill=black,inner sep=1.5pt,label=right:{\footnotesize $\i^{\mathrm{in}}$}] (R\i) at (3, 2.5-0.9*\i) {};
}

\draw[cov] (L1) -- (R2);
\draw[cov] (L2) -- (R1);
\draw[cov] (L3) -- (R4);
\draw[cov] (L4) -- (R5);
\draw[cov] (L5) -- (R3);
\draw[extra] (L1) -- (R5);
\draw[extra] (L3) -- (R2);
\node[font=\footnotesize] at (1.5,-2.6) {bipartite lift $B(D)$; matching in blue};
\end{scope}
\end{tikzpicture}
\caption{The reduction of Lemma~\ref{lem:cyclecover_matching}. Left: a digraph $D$ and a cycle cover---a disjoint union of directed cycles covering every vertex (here $1\!\leftrightarrow\!2$ and $3\!\to\!4\!\to\!5\!\to\!3$). Right: the bipartite lift $B(D)$, with a left copy $v^{\mathrm{out}}$ and a right copy $w^{\mathrm{in}}$ joined whenever $v\to w\in\vec E$. The cover corresponds edge-for-edge to a perfect matching (blue): the permutation $\varphi$ with $v\to\varphi(v)$ in the cover matches $v^{\mathrm{out}}$ to $\varphi(v)^{\mathrm{in}}$. Gray edges of $D$ are not used by this cover.}
\label{fig:cyclecover_matching}
\end{figure}
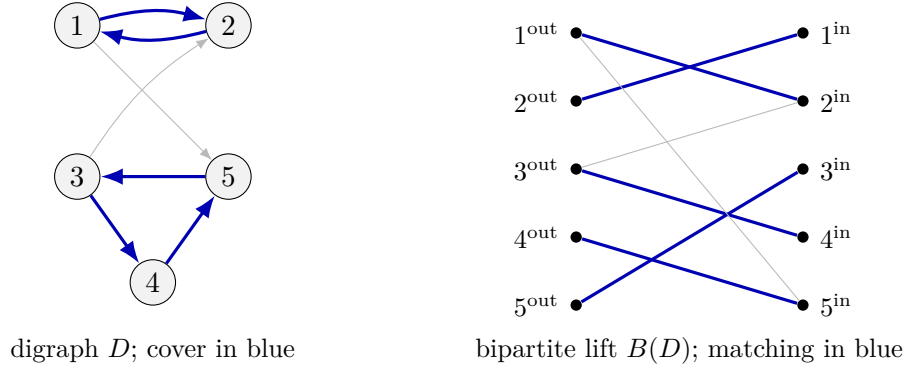

The following lemma, a standard refinement of Hall's theorem, says that when a perfect matching fails to exist, a relatively small obstruction can always be found. We use it repeatedly in Section~\ref{ssec:dense_block_matching}.

\begin{lemma}[Small Hall obstructions]\label{lem:small_hall_witness}
Let $G=(L\sqcup R,E)$ be a bipartite graph with $|L|=|R|=k$. If $G$ has no perfect matching, then there exists a subset $X\subseteq L$ or a subset $Y\subseteq R$ such that $|N(X)|\leq |X|-1 \quad\text{or}\quad |N(Y)|\leq |Y|-1,$ and moreover one may choose such a set of size at most $\lceil k/2\rceil$.
\end{lemma}

\begin{proof}
Since $G$ has no perfect matching, Hall's theorem implies there exists a (left) obstruction $X\subseteq L$ with $|N(X)|\leq |X|-1$. Among all such sets, choose one with $|X|$ minimal.

\smallskip
\noindent{\em Claim:} For this minimal set, we have $|N(X)|=|X|-1$.
Indeed, if $|N(X)|\leq |X|-2$, then for any $x\in X$ we have $N(X\setminus\{x\})\subseteq N(X)$ and hence $|N(X\setminus\{x\})|\ \le\ |N(X)|\ \le\ |X|-2\ =\ |X\setminus\{x\}|-1,$ so $X\setminus\{x\}$ is also an obstruction, contradicting minimality of $|X|$.

\smallskip
If $|X|\leq \lceil k/2\rceil$ we are done. Otherwise $|X|\geq \lceil k/2\rceil+1$, and we define $Y:=R\setminus N(X).$ Using $|N(X)|=|X|-1$, we compute
$$
|Y|\ =\ k-|N(X)|\ =\ k-(|X|-1)\ =\ k-|X|+1\ \le\ \Bigl\lceil\frac{k}{2}\Bigr\rceil.
$$
Moreover, by construction, there are no edges from $X$ to $Y$; hence every neighbor of $Y$ lies in $L\setminus X$, i.e.\ $N(Y)\subseteq L\setminus X$, and therefore $|N(Y)|\ \le\ |L\setminus X|\ =\ k-|X|\ =\ |Y|-1.$ Thus $Y$ is a right obstruction of size at most $\lceil k/2\rceil$, completing the proof.
\end{proof}

\subsection{Polytope tools}\label{ssec:cycle_polytope_lemmas}

We now record the properties of $\vec{\mathcal X}(S)$ for later use: the two descriptions of the polytope announced earlier, and the combinatorial meaning of its {\em rational}  points.

We use the following consequence of strong connectivity repeatedly.

\begin{lemma}\label{lem:proper_expand}
Let $S=(U,F)$ be strongly connected and let $\varnothing\neq I\subsetneq U$. Then $N^+(I)\not\subseteq I$ and $N^-(I)\not\subseteq I$; in particular, $N^+(I)\neq I$ and $N^-(I)\neq I$.
\end{lemma}

\begin{proof}
If $N^+(I)\subseteq I$, no edge of $S$ leaves $I$, so no vertex of the nonempty set $U\setminus I$ is reachable from $I$, contradicting strong connectivity. The case $N^-$ is proved similarly. \end{proof}

\begin{lemma}[Convex-hull description of $\vec{\mathcal X}(S)$]\label{lem:cycle_polytope_conv_cycles}
Let $S=(U,F)$ be a directed graph. The following hold:
\begin{enumerate}
\item[\textup{(i)}] $A^{C}\in\mathcal A(S)$ and $x^{C}\in\vec{\mathcal X}(S)$ for every directed cycle $C$.
\item[\textup{(ii)}] Every $A\in\mathcal A(S)$ can be written as a convex combination of cycle matrices: there exist directed cycles $C_1,\dots,C_M$ in $S$ and weights $\lambda_1,\dots,\lambda_M\geq 0$ with $\sum_{m=1}^M\lambda_m=1$ such that
$$
A\ =\ \sum_{m=1}^M \lambda_m A^{C_m}.
$$
Consequently, every $x\in\vec{\mathcal X}(S)$ is a convex combination of the type frequency vectors of cycles:
$$
\vec{\mathcal X}(S)\ =\ \conv\{\,x^{C}:\ C \text{ is a directed cycle in } S\,\}.
$$
Moreover, the same convex hull is obtained if one restricts to {\em simple} directed cycles.
\end{enumerate}
\end{lemma}

\begin{proof}
(i) By construction, $A^{C}\geq 0$, it is supported only on edges of $C$, and each vertex on the cycle has exactly one unit of outflow and inflow, so $A^{C}\bfo=A^{C\top}\bfo$. Also $\bfo^\top A^{C}\bfo=\sum_{i,j}A^{C}_{ij}=1$. Thus $A^{C}\in\mathcal A(S)$ and $x^C=A^C\bfo\in\vec{\mathcal X}(S)$.

(ii) Fix $A\in\mathcal A(S)$ and view it as a unit circulation on $S$: each directed edge $e=(u_i\to u_j)$ has flow $a_{ij}$. Since $\bfo^\top A\bfo=1$, some entry is positive. Pick any edge with positive flow and follow a directed walk along positive-flow edges; flow conservation ensures that the walk can always be continued. Since $U$ is finite, the walk contains a simple directed cycle $C$ with positive flow on each of its edges.

Let $\gamma:=\min\{a_{ij}:\ (u_i\to u_j)\in C\}>0$ and let $\ell$ be the length of $C$. Then $B:=A-\gamma\ell\,A^{C}$ is still nonnegative, supported on $F$, and remains a circulation (flow conservation is preserved by subtracting a cycle circulation), with total flow $\bfo^\top B\bfo=1-\gamma\ell$. Iterating this cycle-peeling procedure finitely many times (at each step {\em at least one} positive entry is driven to $0$) yields
$$
A\ =\ \sum_{m=1}^M \alpha_m A^{C_m},
\qquad \alpha_m:=\gamma_m\ell_m>0,
\qquad \sum_{m=1}^M \alpha_m = \bfo^\top A\bfo = 1,
$$ which is the desired convex combination. Multiplying on the right by $\bfo$ gives $x=A\bfo=\sum_m \alpha_m x^{C_m}$, hence $\vec{\mathcal X}(S)\subseteq \conv\{x^C\}$. The reverse inclusion follows from (i) and convexity of $\vec{\mathcal X}(S)$.

Finally, any directed cycle decomposes into a multiset (i.e., possibly with repeats) of simple directed cycles, and the associated cycle matrix and type frequency vector are convex combinations of the simple-cycle ones. Thus restricting to simple cycles does not change the convex hull.
\end{proof}

The following characterization expresses membership in $\vec{\mathcal X}(S)$ as a finite system of linear inequalities indexed by subsets of $U$. %It is used in both directions: in necessity arguments to extract a tight linear obstruction when $x^*$ lies on the boundary, and in sufficiency arguments to obtain a uniform margin on a compact subset of the relative interior.

\begin{lemma}[Characterization of $\vec{\mathcal X}(S)$]\label{lem:hall_Xdir}
Assume that $S=(U,F)$ satisfies $N^+(U)=U$ and let $x\in\Delta_q$ (so $x\geq 0$ and $\bfo^\top x=1$). Then the following are equivalent:
\begin{enumerate}
\item[\textup{(i)}] $x\in \vec{\mathcal X}(S)$.
\item[\textup{(ii)}] There exists $A\in\mathcal A(S)$ such that $A\bfo=x$ (equivalently, $A\bfo=A^\top\bfo=x$ and $\bfo^\top A\bfo=1$).
\item[\textup{(iii)}] For every nonempty proper subset $\varnothing\neq I\subsetneq U$,
$$
x_I\ \le\ x_{N^+(I)}.
$$
\end{enumerate}
\end{lemma}
Note that $N^+(U)=U$ follows directly when $S$ is strongly connected and $q\ge2$; and for $q=1$ it says that the unique node of $S$ has a loop. 
\begin{proof}
The equivalence (i)$\Leftrightarrow$(ii) is Definition~\ref{def:cycle_polytope}; we prove (ii)$\Leftrightarrow$(iii).

\xc{(ii)$\Rightarrow$(iii):} Let $A\in\mathcal A(S)$ with $A\bfo=A^\top\bfo=x$. For nonempty $I\subsetneq U$, we have
$$
x_I=\sum_{i\in I}\sum_j a_{ij}=\sum_{j\in N^+(I)}\sum_{i\in I} a_{ij}\leq \sum_{j\in N^+(I)}\sum_i a_{ij}=x_{N^+(I)},
$$
where we used the fact that  $a_{ij}=0$ outside $F$ for the second equality.

\xc{(iii)$\Rightarrow$(ii):} 
We seek a nonnegative matrix $A=[a_{ij}]$, supported on $F$, whose row sums and column sums are both equal to $x$. We realize such a matrix as a flow in a bipartite network.

Introduce a source $s$, left vertices $\ell_1,\dots,\ell_q$, right vertices $r_1,\dots,r_q$, and a sink $t$. Put an edge $s\to\ell_i$ of capacity $x_i$, an edge $r_j\to t$ of capacity $x_j$, and, whenever $(u_i\to u_j)\in F$, an edge $\ell_i\to r_j$ of infinite capacity. Thus, to prove {\em (ii)}, it suffices to find a flow of value $1$: indeed, such a flow saturates all edges leaving $s$ and entering $t$, and the flows on the middle edges then provide the entries of a nonnegative matrix $A$ supported on $F$ whose row and column sums are both equal to $x$.

Since any flow on the network is clearly upper bounded by $1$ (as the maximum flow out from $s$ is 1), it suffices, by the max-flow/min-cut theorem, to show that every $s$--$t$ cut has capacity lower bounded by $1$. Consider such a cut, and let
$$
I:=\{i:\ell_i\text{ lies on the }s\text{-side}\},\qquad J:=\{j:r_j\text{ lies on the }s\text{-side}\}.
$$

If some admissible edge $\ell_i\to r_j$ goes from the $s$-side to the $t$-side, then the cut has infinite capacity, so there is nothing to prove. Thus, for a finite cut, every out-neighbor of $I$ must belong to $J$, that is,
$$
N^+(I)\subseteq J.
$$
The finite-capacity edges crossing the cut are therefore $s\to\ell_i$ for $i\notin I$ and $r_j\to t$ for $j\in J$. Hence the capacity of the cut is
$$
\sum_{i\notin I}x_i+\sum_{j\in J}x_j=(1-x_I)+x_J\geq (1-x_I)+x_{N^+(I)}.
$$

If $I=\varnothing$, this is at least $1$. If $\varnothing\neq I\subsetneq U$, hypothesis {\em (iii)} gives $x_{N^+(I)}\geq x_I$, so the capacity is again at least $1$. Finally, if $I=U$, the assumption $N^+(U)=U$ and the inclusion $N^+(I)\subseteq J$ force $J=U$, so the capacity is exactly $1$. Thus every $s$--$t$ cut has capacity at least $1$. On the other hand, the cut whose $s$-side consists only of $s$ has capacity $\sum_i x_i=1$. Hence the minimum cut has capacity $1$, and the max-flow/min-cut theorem yields a maximum flow $f$ of value $1$.

Since the total capacity of the edges leaving $s$ is $1$, every edge $s\to\ell_i$ is saturated, so $f(s,\ell_i)=x_i$ for every $i$. Similarly, every edge $r_j\to t$ is saturated, so $f(r_j,t)=x_j$ for every $j$. Now define
$$
a_{ij}:=
\begin{cases}
f(\ell_i,r_j), & (u_i\to u_j)\in F,\\
0, & \text{otherwise}.
\end{cases}
$$
Then $A=[a_{ij}]$ is nonnegative and supported on $F$. Flow conservation at $\ell_i$ and $r_j$ gives
$$
\sum_j a_{ij}=x_i,\qquad \sum_i a_{ij}=x_j,
$$
and therefore $A\bfo=A^\top\bfo=x$. Moreover, $\bfo^\top A\bfo=\bfo^\top x=1$.  Thus $A\in\mathcal A(S)$ and $A\bfo=x$, proving {\em (ii)}.
\end{proof}

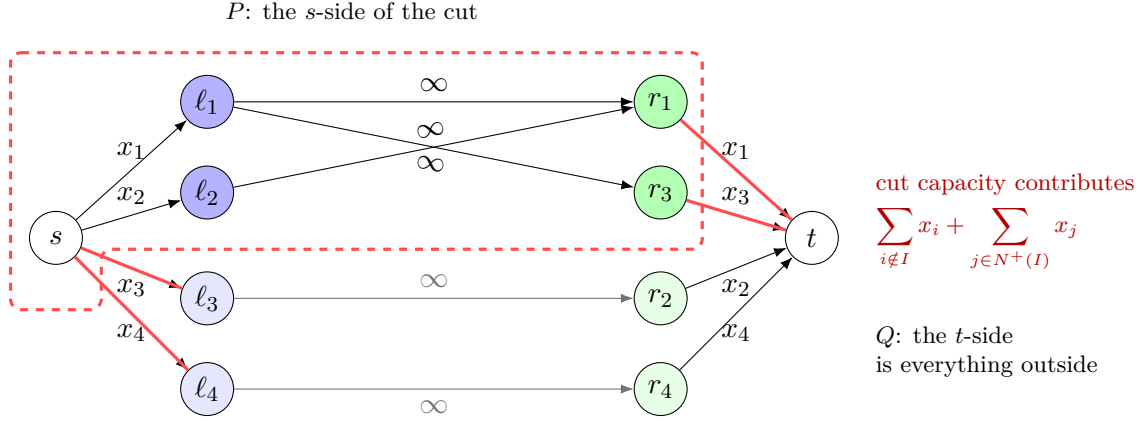
\begin{figure}[t]
\centering
\begin{tikzpicture}[
>=Latex,
node distance=10mm and 12mm,
every node/.style={font=\small},
src/.style={circle, draw, minimum size=7mm, inner sep=0pt},
typeL/.style={circle, draw, minimum size=7mm, inner sep=0pt, fill=blue!10},
typeR/.style={circle, draw, minimum size=7mm, inner sep=0pt, fill=green!10},
inI/.style={circle, draw, minimum size=7mm, inner sep=0pt, fill=blue!30},
inN/.style={circle, draw, minimum size=7mm, inner sep=0pt, fill=green!30},
lab/.style={font=\scriptsize},
cutbox/.style={rounded corners, dashed, very thick, draw=red!70}
]

\node[src] (s) at (0,0) {$s$};
\node[src] (t) at (10,0) {$t$};

\node[inI] (i1) at (2, 1.8) {$\ell_1$};
\node[inI] (i2) at (2, 0.6) {$\ell_2$};
\node[typeL] (i3) at (2,-0.8) {$\ell_3$};
\node[typeL] (i4) at (2,-2.0) {$\ell_4$};

\node[inN] (j1) at (8, 1.8) {$r_1$};
\node[inN] (j3) at (8, 0.6) {$r_3$};
\node[typeR] (j2) at (8,-0.8) {$r_2$};
\node[typeR] (j4) at (8,-2.0) {$r_4$};

\draw[->, ] (s) -- (i1) node[midway, above] {$x_1$};
\draw[->, ] (s) -- (i2) node[midway, above] {$x_2$};
\draw[->, ] (s) -- (i3) node[midway, below] {$x_3$};
\draw[->, ] (s) -- (i4) node[midway, below] {$x_4$};

\draw[->, ] (j1) -- (t) node[midway, above] {$x_1$};
\draw[->, ] (j3) -- (t) node[midway, above] {$x_3$};
\draw[->, ] (j2) -- (t) node[midway, below] {$x_2$};
\draw[->, ] (j4) -- (t) node[midway, below] {$x_4$};

\draw[->, ] (i1) -- (j1) node[midway, above, sloped] {$\infty$};
\draw[->, ] (i1) -- (j3) node[midway, below, sloped] {$\infty$};
\draw[->, ] (i2) -- (j1) node[midway, above, sloped] {$\infty$};

\draw[->,, opacity=.55] (i3) -- (j2) node[midway, above, sloped] {$\infty$};
\draw[->,, opacity=.55] (i4) -- (j4) node[midway, below, sloped] {$\infty$};

\draw[cutbox]
(-0.6, 2.45) -- (8.55, 2.45) -- (8.55, -0.15) -- (0.6, -0.15) -- (0.6, -0.95) -- (-0.6, -0.95)
-- cycle;

\node[lab, anchor=south west] at ($(i1)+(0.1,0.95)$)
{$P$: the $s$-side of the cut};
\node[lab, align=left, anchor=west] at (10.7,-1.5) {$Q$: the $t$-side\\is everything outside};

\draw[red!70, very thick] (s) -- (i3);
\draw[red!70, very thick] (s) -- (i4);
\draw[red!70, very thick] (j1) -- (t);
\draw[red!70, very thick] (j3) -- (t);

\node[lab, text=red!70!black, align=left, anchor=west] at (10.7,0.2)
{cut capacity contributes\\[1mm]
$\displaystyle \sum_{i\notin I} x_i + \sum_{j\in N^+(I)} x_j$};

\end{tikzpicture}
\caption{Flow network for the proof of Lemma~\ref{lem:hall_Xdir}. The dashed curve indicates the $s$-side $P$ of a finite $s$--$t$ cut. If $I$ is the set of left nodes in $P$ and $J$ the set of right nodes in $P$, finiteness of the cut forces $N^+(I)\subseteq J$, since every admissible edge $\ell_i\to r_j$ has infinite capacity. Thus the cut capacity is $(1-x_I)+x_J\geq (1-x_I)+x_{N^+(I)}$. In the example shown, $I=\{1,2\}$ and $J=N^+(I)=\{1,3\}$, and the red edges are precisely the finite-capacity edges crossing the cut.
}
\label{fig:hallXdir_flow_cut}
\end{figure}

We now introduce the following definition, see Figure~\ref{fig:blowup_example} for an illustration:
\begin{definition}[Complete blow-up]\label{def:complete_blowup}
For $m=(m_1,\dots,m_q)\in\Z_{\ge0}^q$, the {\em complete blow-up} $S[m]$ is the digraph obtained by replacing each node $u_i$ of $S$ by a set $V_i$ of $m_i$ vertices and putting in every directed edge from $V_i$ to $V_j$ whenever $(u_i\to u_j)\in F$. A loop at $u_i$ implies that  every edge inside $V_i$ is present, loops included.
\end{definition}
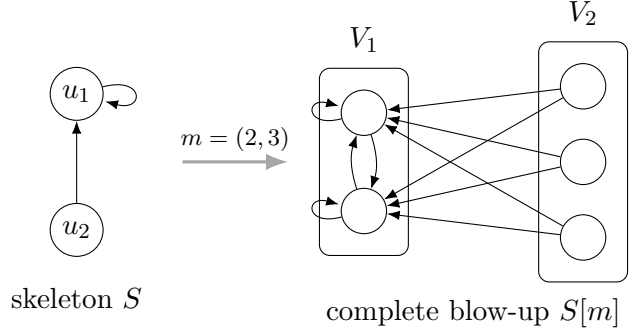
\begin{figure}[t]
\centering
\begin{tikzpicture}[
>=Latex,
every node/.style={font=\small},
type/.style={circle,draw,minimum size=7mm,inner sep=0pt},
vtx/.style={circle,draw,minimum size=6mm,inner sep=0pt},
lab/.style={font=\scriptsize}
]

% Skeleton
\begin{scope}
\node[type] (u1) at (0,0.9) {$u_1$};
\node[type] (u2) at (0,-0.9) {$u_2$};

\draw[->] (u1) edge[loop right] ();
\draw[->] (u2) -- (u1);

\node at (0,-1.8) {skeleton $S$};
\end{scope}

% Blow-up arrow
\draw[->,very thick,gray!70] (1.4,0) -- (2.8,0);
\node[lab,above] at (2.1,0) {$m=(2,3)$};

% Complete blow-up
\begin{scope}[xshift=5cm]

% V_1
\node[vtx] (a1) at (-1.2,0.65) {};
\node[vtx] (a2) at (-1.2,-0.65) {};

\node[draw,rounded corners,
      fit=(a1)(a2),
      inner xsep=8pt,inner ysep=8pt] (B1) {};

\node[above=2pt of B1] {$V_1$};

% loops and both directed edges inside V_1
\draw[->] (a1) edge[loop left] ();
\draw[->] (a2) edge[loop left] ();
\draw[->] (a1) to[bend left=18] (a2);
\draw[->] (a2) to[bend left=18] (a1);

% V_2
\node[vtx] (b1) at (1.7,1.0) {};
\node[vtx] (b2) at (1.7,0) {};
\node[vtx] (b3) at (1.7,-1.0) {};

\node[draw,rounded corners,
      fit=(b1)(b2)(b3),
      inner xsep=8pt,inner ysep=8pt] (B2) {};

\node[above=2pt of B2] {$V_2$};

% Every edge from V_2 to V_1
\foreach \a in {b1,b2,b3} {
  \foreach \b in {a1,a2} {
    \draw[->] (\a) -- (\b);
  }
}

\node at (0.25,-2.0) {complete blow-up $S[m]$};

\end{scope}

\end{tikzpicture}
\caption{The skeleton $S$ and its complete blow-up $S[m]$ for $m=(2,3)$. The loop at $u_1$ produces all directed edges inside $V_1$, including loops, while the edge $u_2\to u_1$ produces all directed edges from $V_2$ to $V_1$.}
\label{fig:blowup_example}
\end{figure}

The inequalities of Lemma~\ref{lem:hall_Xdir} have a concrete combinatorial meaning: we say that a vector of type counts is {\em feasible} if it lies in $\vec{\mathcal X}(S)$ after normalization; feasibility is exactly the statement that the corresponding {\em blow-up} of the skeleton can be partitioned into disjoint cycles.

\begin{proposition}[Cycle covers in complete blow-ups]\label{prop:blowup_cover}
Let $m\in\Z_{\ge0}^q$ and $k:=\bfo^\top m\ge1$. Then $S[m]$ admits a cycle cover if and only if $m/k\in\vec{\mathcal X}(S)$.
\end{proposition}

The ``if'' direction of the above proposition is also proved as~\cite[Theorem~5]{chen2025hamiltonicity}; we include a short proof based on Hall's theorem for completeness.

\begin{proof}
Suppose first that $S[m]$ admits a cycle cover, and let $n_{ij}$ be the number of edges from $V_i$ to $V_j$ in the cover. The matrix $N=[n_{ij}]$ is supported on $F$, and since every vertex of the cover has outdegree and indegree $1$, its row sums and column sums both equal $m$. Hence $A:=N/k$ lies in $\mathcal A(S)$ and $A\bfo=m/k$, so $m/k\in\vec{\mathcal X}(S)$.

Conversely, suppose $m/k\in\vec{\mathcal X}(S)$, and consider the bipartite lift $B$ of $S[m]$. For a set $Z$ of left vertices of $B$, let $N_B(Z)$ denote its set of neighbors on the right, and let $I\subseteq U$ be the set of nodes having a copy in $Z$. By the definition of $S[m]$, if $i\in I$ and $u_i\to u_j$ is an edge of $S$, then every vertex of $V_i$ is joined to every vertex of $V_j$. Hence
$$
N_B(Z)=\bigcup_{j\in N^+(I)}V_j^{\rm in},
$$
where $V_j^{\rm in}$ denotes the right-side copy of $V_j$ in the bipartite lift. Consequently, $|N_B(Z)|=m_{N^+(I)}$.

On the other hand, since $Z$ contains only copies of nodes in $I$, $|Z|\leq m_I$.  Since $m/k\in\vec{\mathcal X}(S)$, Lemma~\ref{lem:hall_Xdir} gives $m_I\leq m_{N^+(I)}$ (for $I=U$ this holds trivially, since $N^+(U)=U$). Putting these inequalities together, we get
$$
|Z|\leq m_I\leq m_{N^+(I)}=|N_B(Z)|.
$$
Since $Z$ was arbitrary, Hall's condition holds for every set $Z$ of left vertices and Hall's theorem gives a perfect matching of $B$, and Lemma~\ref{lem:cyclecover_matching} then gives a cycle cover of $S[m]$.
\end{proof}

The preceding proposition now yields the corresponding statement for sampled graphs: if $D_n[V']$ admits a cycle cover, then so does the complete blow-up determined by the type counts of $V'$, and hence $x(V')\in\vec{\mathcal X}(S)$:

\begin{lemma}\label{lem:cover_implies_Xdir}
Let $D_n\sim W$ with skeleton $S=(U,F)$, and let $V'\subseteq V$ be nonempty. If $D_n[V']$ admits a cycle cover, then $x(V')\in\vec{\mathcal X}(S)$.
\end{lemma}

\begin{proof}

Let $m$ be the type counts of $V'$. Then $\bfo^\top m=|V'|$ and $x(V')=m/|V'|$. Since $S[m]$ admits a cycle cover (take the vertex set of $S[m]$ to be $V'$ itself, the block $V_i$ consisting of the vertices of $V'$ of type $u_i$; since $\pi$ is a digraph homomorphism, every edge $v\to w$ of $D_n[V']$ satisfies $(\pi(v)\to\pi(w))\in F$ and is thus an edge of $S[m]$, so $D_n[V']$ is a spanning subgraph of $S[m]$ and any cycle cover of $D_n[V']$ is also a cycle cover of $S[m]$), Proposition~\ref{prop:blowup_cover} gives $m/|V'|\in\vec{\mathcal X}(S)$, and hence $x(V')\in\vec{\mathcal X}(S)$.
\end{proof}

\subsection{Realizable supports and Condition D}\label{ssec:Dgeometric}

Condition D is a statement about directed cycles in the skeleton. We record here what it gives in the cycle polytope: a description of the subsets of $U$ that occur as supports of points of $\vec{\mathcal X}(S)$, and the dimension of the faces $F_{T_r}$ along a nested sequence. These are the forms in which the condition will enter the arguments later on.

\begin{proposition}[Realizable supports]\label{prop:realizable_supports}
Let $T\subseteq U$ be nonempty.
\begin{enumerate}
\item[\textup{(i)}] There exists $x\in\vec{\mathcal X}(S)$ with $\supp(x)=T$ if and only if every node of $S[T]$ lies on a directed cycle of $S[T]$.
\item[\textup{(ii)}] Assume Condition D holds with ordering $t_1,\dots,t_q$, and set $T_r:=\{t_1,\dots,t_r\}$. Then $F_{T_r}:=\vec{\mathcal X}(S)\cap H_{T_r}$ is full-dimensional in $H_{T_r}$ for every $r$; equivalently, $\rint_{H_{T_r}}(F_{T_r})\neq\varnothing$.
\end{enumerate}
\end{proposition}

\begin{proof}
For (i), we have from Lemma~\ref{lem:cycle_polytope_conv_cycles} that any point of the polytope is a convex combination of type frequency vectors of cycles of $S$. If the point is supported on $T$, every cycle appearing with positive weight in the convex combination visits only nodes of $T$ and is therefore a cycle of $S[T]$; if moreover the support is all of $T$, each $u\in T$ lies on one of them. Conversely, pick for each $u\in T$ a directed cycle of $S[T]$ through $u$ and average the type frequency vectors of the chosen cycles: the result lies in $\vec{\mathcal X}(S)$ and has support exactly $T$.

 For (ii), Condition D and part \textup{(i)} together provide points $x^{(1)},\dots,x^{(r)}\in F_{T_r}$ with $\supp(x^{(s)})=T_s$. Write them in the coordinates $t_1,\dots,t_r$ as the columns of an $r\times r$ matrix. Since $x^{(s)}$ vanishes on $t_{s+1},\dots,t_r$ and is positive at $t_s$, that matrix is upper triangular with strictly positive diagonal, so the $r$ points are linearly independent; lying in the hyperplane $\{\bfo^\top x=1\}$, they are then affinely independent. Their convex hull is an $(r-1)$-dimensional simplex contained in $F_{T_r}$, and $H_{T_r}$ has dimension $r-1$, so $F_{T_r}$ is full-dimensional in $H_{T_r}$.
\end{proof}

Condition D does not require $S[T_r]$ itself to admit a cycle cover:
each node need only lie on some directed cycle of $S[T_r]$.
For example, let $S$ have the edges $1\to1$, $1\leftrightarrow2$
and $1\leftrightarrow3$. Condition D holds for the ordering $1,2,3$,
but $S$ has no cycle cover, since $N^+(\{2,3\})=\{1\}$.
Giving node $1$ a second copy removes this obstruction:
every prefix of the ordering $1_a,2,1_b,3$ has a cycle cover
in the blow-up $S[(2,1,1)]$; see Figure~\ref{fig:D_without_skeleton_cover}.

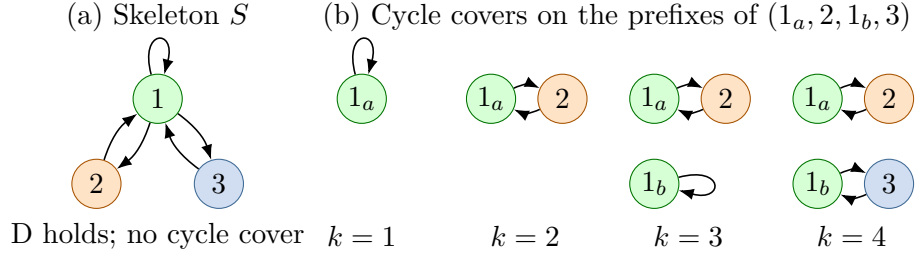
\begin{figure}[t]
\centering
\begin{tikzpicture}[scale=.9,
  >=Latex,
  every node/.style={font=\small},
  vertex/.style={circle,draw,minimum size=6.5mm,inner sep=0pt},
  typeone/.style={vertex,draw=green!45!black,fill=green!16},
  typetwo/.style={vertex,draw=orange!65!black,fill=orange!22},
  typethree/.style={vertex,draw=RoyalBlue!65!black,fill=RoyalBlue!16},
  arc/.style={->,semithick},
  panel/.style={font=\small,align=center}
]
% The skeleton: the two cycles share node 1.
\node[panel] at (1.05,2.10) {(a) Skeleton $S$};
\node[typeone] (s1) at (1.05,0.90) {$1$};
\node[typetwo] (s2) at (0.15,-0.35) {$2$};
\node[typethree] (s3) at (1.95,-0.35) {$3$};
\path[arc] (s1) edge[loop above,min distance=7mm] (s1);
\path[arc] (s1) edge[bend left=18] (s2)
                  edge[bend left=18] (s3)
            (s2) edge[bend left=18] (s1)
            (s3) edge[bend left=18] (s1);
\node[panel] at (1.05,-1.12) {D holds; no cycle cover};

%\draw[gray!40] (2.85,-1.38) -- (2.85,2.32);
\node[panel] at (7.80,2.10)
  {(b) Cycle covers on the prefixes of $(1_a,2,1_b,3)$};

% k = 1. A loop is the entire cover.
\begin{scope}[xshift=3.55cm]
  \node[typeone] (a1) at (0.5,0.90) {$1_a$};
  \path[arc] (a1) edge[loop above,min distance=7mm] (a1);
  \node[panel] at (0.50,-1.12) {$k=1$};
\end{scope}

% k = 2. Replace the loop by the 2-cycle.
\begin{scope}[xshift=5.95cm]
  \node[typeone] (b1) at (0,0.90) {$1_a$};
  \node[typetwo] (b2) at (1.05,0.90) {$2$};
  \path[arc] (b1) edge[bend left=25] (b2)
              (b2) edge[bend left=25] (b1);
  \node[panel] at (0.50,-1.12) {$k=2$};
\end{scope}

% k = 3. The new copy of type 1 contributes its loop.
\begin{scope}[xshift=8.35cm]
  \node[typeone] (c1) at (0,0.90) {$1_a$};
  \node[typetwo] (c2) at (1.05,0.90) {$2$};
  \node[typeone] (c3) at (0,-0.35) {$1_b$};
  \path[arc] (c1) edge[bend left=25] (c2)
              (c2) edge[bend left=25] (c1);
  \path[arc] (c3) edge[loop right,min distance=7mm] (c3);
  \node[panel] at (0.50,-1.12) {$k=3$};
\end{scope}

% k = 4. Two disjoint 2-cycles cover the entire blow-up.
\begin{scope}[xshift=10.75cm]
  \node[typeone] (d1) at (0,0.90) {$1_a$};
  \node[typetwo] (d2) at (1.05,0.90) {$2$};
  \node[typeone] (d3) at (0,-0.35) {$1_b$};
  \node[typethree] (d4) at (1.05,-0.35) {$3$};
  \path[arc] (d1) edge[bend left=25] (d2)
              (d2) edge[bend left=25] (d1)
              (d3) edge[bend left=25] (d4)
              (d4) edge[bend left=25] (d3);
  \node[panel] at (0.50,-1.12) {$k=4$};
\end{scope}
\end{tikzpicture}
\caption{Condition D need not give a cycle cover of the skeleton.
The skeleton in (a) satisfies D for the ordering $1,2,3$, but nodes
$2$ and $3$ have the same sole out-neighbor. In (b), a second copy of
node $1$ allows a cycle cover on every prefix of $1_a,2,1_b,3$.
Only edges of the chosen covers are drawn in (b); the vertex sets
are nested, while the covers may change.}
\label{fig:D_without_skeleton_cover}
\end{figure}

%\begin{example}[Condition D does not require nested cycle covers in the skeleton]\label{ex:D_not_skeleton_S}
%Let $S$ have the edges $1\to1$, $1\leftrightarrow2$ and $1\leftrightarrow3$. Condition D holds for the ordering $1,2,3$, certified by the loop at $1$, then the cycle $1\leftrightarrow2$, then the cycle $1\leftrightarrow3$. The skeleton itself, however, has no cycle cover at all: $N^+(\{2,3\})=\{1\}$, so types $2$ and $3$ compete for the same single out-neighbor. Giving type $1$ a second copy removes the obstruction. Writing $1_a,1_b$ for the two copies, the ordering $1_a,2,1_b,3$ has a cycle cover on every prefix --- a loop at $1_a$; the cycle $1_a\leftrightarrow2$; that cycle together with a loop at $1_b$; and finally the two disjoint cycles $1_a\leftrightarrow2$ and $1_b\leftrightarrow3$.
%\end{example}

\subsection{Balanced orderings}\label{ssec:steinitz_balancing}

The next tool is a {\em Steinitz} rearrangement: the multiset of unit vectors specified by a count vector can be ordered so that every prefix sum stays within bounded $\ell_\infty$-distance of the segment from the origin to that vector. It is used to append vertices in a balanced order in both sufficiency proofs.

\begin{lemma}[Balanced ordering]
\label{lem:balanced_prefix_rounding_aux}\label{lem:appendable_balancing} There exists a constant $C_{\mathrm{bal}}(q)<\infty$, depending only on $q$, such that the following holds:  Let $d=(d_1,\dots,d_q)\in\Z_{\geq 0}^q$ and $|d|:=\bfo^\top d\geq 1.$ Then there exist vectors $u^{(0)},u^{(1)},\dots,u^{(|d|)}\in\Z_{\geq 0}^q$ such that
$$
u^{(0)}=0,\qquad u^{(|d|)}=d,\qquad
u^{(s+1)}-u^{(s)}\in\{e_1,\dots,e_q\}\quad(0\leq s<|d|),
$$ 
and $\left\|u^{(s)}-\frac{s}{|d|}\,d\right\|_\infty\leq C_{\mathrm{bal}}(q)$, for all $0\leq s\leq |d|$.
\end{lemma}

\begin{proof}
For each $i\in[q]$, set $v_i:=e_i-d/|d|\in\R^q$ and let $\mathcal M$ be the multiset containing $d_i$ copies of $v_i$, for $i=1,\dots,q$. It has $\bfo^\top d=|d|$ elements, and
$$
\|v_i\|_\infty\leq 1,
\qquad
\sum_{i=1}^q d_i v_i
=
\sum_{i=1}^q d_ie_i-|d|\,\frac{d}{|d|}
=
d-d
=
0,
$$ where the norm bound holds because every coordinate of $d/|d|$ lies in $[0,1]$. By the Steinitz rearrangement theorem in $\ell_\infty^q$~\cite{SteinitzBedingtKR,grinbergSevastyanov1980}, there exists a constant $C_{\mathrm{bal}}(q)<\infty$ depending only on $q$ and an ordering $v^{(1)},\dots,v^{(|d|)}$ of $\mathcal M$ such that every partial sum satisfies
$$
\left\|\sum_{r=1}^s v^{(r)}\right\|_\infty \leq C_{\mathrm{bal}}(q)
\qquad \mbox{ for } s\in[|d|].
$$
Let $\tau_r$ be the type corresponding to $v^{(r)}$, so that $v^{(r)}=e_{\tau_r}-d/|d|$, and set $u^{(s)}:=\sum_{r=1}^s e_{\tau_r}$ (with $u^{(0)}:=0$). Then for every $s$,
$$
\sum_{r=1}^s v^{(r)}
=
\sum_{r=1}^s \Bigl(e_{\tau_r}-\frac{d}{|d|}\Bigr)
=
u^{(s)}-\frac{s}{|d|}\,d.
$$
Hence $\|u^{(s)}-(s/|d|)d\|_\infty\leq C_{\mathrm{bal}}(q)$, and $u^{(|d|)}=d$ since the multiset contains exactly $d_i$ copies of $e_i$, which is the claim.
\end{proof}

\section{Building nested sequences of cycle covers}\label{sec:chains}

Both sufficiency proofs need a nested sequence of vertex sets whose type frequency vectors stay in a fixed compact subset of $\rint\vec{\mathcal X}(S)$, and a way to ensure that the random graph covers every set of the sequence at once. The second requirement is probabilistic and is common to the two proofs; it is handled in  Section~\ref{ssec:dense_block_matching}. The sequence of sets, however, differs. For $\mathcal N$, the sets are the prefixes of a balanced ordering of all $n$ vertices. For $\mathcal S$ the sequence must start with a single looped vertex, and building its first part  is where Condition D is used; this is the subject of Section~\ref{ssec:count_paths_templates}.

\subsection{From Condition D to a finite digraph}\label{ssec:count_paths_templates}

Condition D provides a nested sequence of support faces, each one full-dimensional inside its own coordinate subspace (Proposition~\ref{prop:realizable_supports}). We first construct a sequence of integer vectors, that grows by one unit in one coordinate at a time, and whose supports follow the ordering of Condition D, and whose normalizations never leave the polytope. Figure~\ref{fig:seed_path} shows the shape of the resulting path.

\begin{lemma}\label{lem:integer_path}
Let $P\subseteq\Delta_q$ be a rational polytope. Suppose there is an ordering $t_1,\dots,t_q$ of $[q]$ such that, with $T_r:=\{t_1,\dots,t_r\}$, the set $P\cap H_{T_r}$ is full-dimensional in $H_{T_r}$ for every $r$. Then there exist $Q\in\N$ and integer vectors $m^{(1)},\dots,m^{(Q)}\in\Z_{\ge0}^q$ such that
$$
m^{(1)}=e_{t_1},\qquad \bfo^\top m^{(k)}=k,\qquad m^{(k+1)}-m^{(k)}\in\{e_1,\dots,e_q\},
$$
$$
\frac{m^{(k)}}{k}\in P\quad(1\leq k\leq Q),\qquad \frac{m^{(Q)}}{Q}\in\rint(P)\cap\mathbb Q^q.
$$
\end{lemma}

In order to keep the argument simple, the proof proceeds in the cone over $P$ rather than in $P$ itself, which avoids  renormalization at each step. It then suffices to check that the path never crosses a facet of the cone.

\begin{proof}
Call a finite sequence of vectors in $\Z_{\ge0}^q$ a {\em unit step path} if each vector is obtained from the preceding one by adding a canonical basis vector. Let $\operatorname{cone}(P):=\{\lambda x:\lambda\ge0,\ x\in P\}$. Since $\bfo^\top x=1$ for every $x\in P$, a nonzero $m\in\Z_{\ge0}^q$ lies in $\operatorname{cone}(P)$ if and only if $m/(\bfo^\top m)\in P$. It therefore suffices to construct a unit step path in $\operatorname{cone}(P)$ that starts at $e_{t_1}$ and ends at $Qy$ for some $Q\in\N$ and some $y\in\rint(P)\cap\mathbb Q^q$. Indeed, the $k$th vector $m^{(k)}$ of such a path has coordinate sum $k$, so $m^{(k)}/k\in P$, and $\bfo^\top y=1$ gives $m^{(Q)}/Q=y$.

For $r\in[q]$, let $P_r:=P\cap H_{T_r}$ and $C_r:=\{\lambda x:\lambda\ge0,\ x\in P_r\}$, and identify $\R^{T_r}$ with the set of vectors in $\R^q$ supported in $T_r$. Since $P_r$ is a rational polytope of dimension $r-1$ in $H_{T_r}$, the cone $C_r$, which is generated by the vertices of $P_r$, is a rational polyhedral cone of full dimension in $\R^{T_r}$, and $C_r\subseteq\operatorname{cone}(P)$. Fix $y^{(r)}\in\rint_{H_{T_r}}(P_r)\cap\mathbb Q^q$. Then $y^{(r)}$ is an interior point of $C_r$ in $\R^{T_r}$. Moreover, $y^{(r)}_i>0$ for every $i\in T_r$: the sets $P_r\subseteq\Delta_q\cap H_{T_r}$ both have affine hull $H_{T_r}$, so $y^{(r)}$ lies in the relative interior of $\Delta_q\cap H_{T_r}$. Finally, $H_{T_1}=\{e_{t_1}\}$, so $y^{(1)}=e_{t_1}$.

We claim that, for each $2\leq r\leq q$, there is an integer $M_r$ with the following property: if $N\geq M_r$ is an integer with $Ny^{(r-1)}\in\Z^q$, then for every sufficiently large integer $N'$ with $N'y^{(r)}\in\Z^q$ {\em there is a unit step path} in $C_r$ from $Ny^{(r-1)}$ to $N'y^{(r)}$.

To prove the claim, fix $r$ and write $t:=t_r$, $y:=y^{(r-1)}$ and $z:=y^{(r)}$. Choose nonzero integral linear functionals $\ell_1,\dots,\ell_s$ on $\R^{T_r}$ such that $C_r=\{v\in\R^{T_r}:\ell_j(v)\ge0\text{ for all }j\}$. Since $z$ is an interior point of $C_r$, we have $\ell_j(z)>0$ for every $j$. Let $L:=\max_j\sum_{i\in T_r}|\ell_j(e_i)|$, so that $|\ell_j(w)|\leq L\|w\|_\infty$ for all $w\in\R^{T_r}$, and fix an integer $\beta>L\,C_{\mathrm{bal}}(q)$.

First, $\ell_j(e_t)\ge1$ whenever $\ell_j(y)=0$. Indeed, $\ell_j$ is nonnegative on $C_{r-1}$, because $C_{r-1}\subseteq C_r$, and it vanishes at $y$, which is an interior point of $C_{r-1}$ in $\R^{T_{r-1}}$. For every $w\in\R^{T_{r-1}}$ and every sufficiently small $\varepsilon>0$, both $y+\varepsilon w$ and $y-\varepsilon w$ lie in $C_{r-1}$, so $\ell_j(w)\ge0$ and $-\ell_j(w)\ge0$. Hence $\ell_j$ vanishes on $\R^{T_{r-1}}$, and since $z-z_te_t\in\R^{T_{r-1}}$,
$$
0<\ell_j(z)=z_t\,\ell_j(e_t).
$$
As $z_t>0$ and $\ell_j(e_t)$ is an integer, $\ell_j(e_t)\ge1$.

Now choose $M_r$ such that $M_r\,\ell_j(y)\geq \beta\bigl(1+|\ell_j(e_t)|\bigr)$ for every $j$ with $\ell_j(y)>0$. Let $N\geq M_r$ be an integer with $Ny\in\Z^q$, and put $a:=Ny+\beta e_t$. For $0\leq c\leq \beta$,
$$
\ell_j(Ny+ce_t)=N\ell_j(y)+c\,\ell_j(e_t)\ge
\begin{cases}
c & \text{if }\ell_j(y)=0,\\
\beta & \text{if }\ell_j(y)>0,
\end{cases}
$$
by the previous paragraph in the first case and by the choice of $M_r$ in the second. Hence $Ny,\,Ny+e_t,\,\dots,\,a$ is a unit path in $C_r$, and $\ell_j(a)\geq \beta$ for every $j$.

Let $N'$ be an integer with $N'z\in\Z^q$, large enough that $N'z_i>a_i$ for all $i\in T_r$ and $N'\ell_j(z)\geq \beta$ for all $j$; all sufficiently large such $N'$ qualify, because $z_i>0$ for $i\in T_r$ and $\ell_j(z)>0$ for all $j$. Put $b:=N'z$ and $d:=b-a$. Then $d\in\Z_{\ge0}^q$ is supported in $T_r$ and $|d|:=\bfo^\top d\ge1$, so Lemma~\ref{lem:balanced_prefix_rounding_aux} yields a unit path $0=u^{(0)},u^{(1)},\dots,u^{(|d|)}=d$ with
$$
\Bigl\|u^{(s)}-\frac{s}{|d|}\,d\Bigr\|_\infty\leq C_{\mathrm{bal}}(q)\qquad(0\leq s\leq |d|).
$$
Since $0\leq u^{(s)}\leq d$, each $u^{(s)}$ is supported in $T_r$. The point $a+(s/|d|)d$ is a convex combination of $a$ and $b$, and $\ell_j(a),\ell_j(b)\geq \beta$, so $\ell_j\bigl(a+(s/|d|)d\bigr)\geq \beta$. Therefore
$$
\ell_j\bigl(a+u^{(s)}\bigr)
\geq \beta-L\,\Bigl\|u^{(s)}-\frac{s}{|d|}\,d\Bigr\|_\infty
\geq \beta-L\,C_{\mathrm{bal}}(q)>0
$$
for all $j$ and $s$, that is, $a+u^{(0)},\dots,a+u^{(|d|)}$ is a unit path in $C_r$ from $a$ to $b$. Appending it to the unit path from $Ny$ to $a$ gives a unit path in $C_r$ from $Ny$ to $N'z$. This proves the claim.

It remains to concatenate these paths. Choose an integer $N_1\geq M_2$ (any $N_1\ge1$ if $q=1$); then $e_{t_1},2e_{t_1},\dots,N_1e_{t_1}$ is a unit path in $C_1$ ending at $N_1y^{(1)}$. For $r=2,\dots,q$ in turn, apply the claim with $N:=N_{r-1}$ to obtain a unit path in $C_r$ from $N_{r-1}y^{(r-1)}$ to $N_ry^{(r)}$, where $N_r$ is an integer with $N_ry^{(r)}\in\Z^q$ that we may take to satisfy $N_r\geq M_{r+1}$ if $r<q$. Concatenating these paths gives a unit path in $C_1\cup\dots\cup C_q\subseteq\operatorname{cone}(P)$ from $e_{t_1}$ to $N_qy^{(q)}$. Since $T_q=[q]$, we have $P_q=P$, and since $P$ is full-dimensional in $H_{[q]}$, the relative interior of $P$ in $H_{[q]}$ is $\rint(P)$. Thus $y^{(q)}\in\rint(P)\cap\mathbb Q^q$, and the first paragraph, applied with $Q:=N_q$ and $y:=y^{(q)}$, completes the proof.
\end{proof}

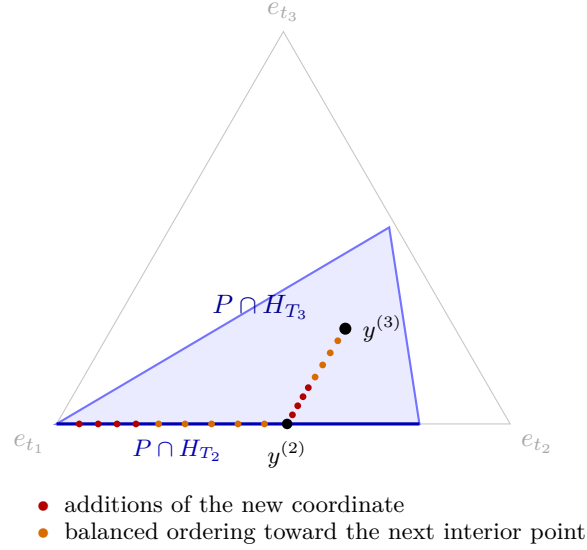
\begin{figure}[t]
\centering
\begin{tikzpicture}[>=Latex, every node/.style={font=\small}]
\coordinate (e1) at (0,0);
\coordinate (e2) at (6,0);
\coordinate (e3) at (3,5.19);
\draw[gray!45] (e1)--(e2)--(e3)--cycle;
\node[below left,gray!70] at (e1) {\footnotesize $e_{t_1}$};
\node[below right,gray!70] at (e2) {\footnotesize $e_{t_2}$};
\node[above,gray!70] at (e3) {\footnotesize $e_{t_3}$};
\coordinate (v2) at (4.8,0);
\coordinate (v3) at (4.4,2.6);
\fill[blue!8] (e1)--(v2)--(v3)--cycle;
\draw[blue!55,thick] (e1)--(v2)--(v3)--cycle;
\node[blue!55!black,font=\footnotesize] at (2.7,1.55) {$P\cap H_{T_3}$};
\draw[blue!70!black, very thick] (e1)--(v2);
\node[blue!70!black,font=\scriptsize] at (1.6,-0.35) {$P\cap H_{T_2}$};
\foreach \x in {0.30,0.55,0.80,1.05} \node[circle,fill=red!70!black,inner sep=0.9pt] at (\x,0) {};
\foreach \x in {1.35,1.70,2.05,2.40,2.75} \node[circle,fill=orange!85!black,inner sep=0.9pt] at (\x,0) {};
\node[circle,fill=black,inner sep=1.4pt,label={[font=\scriptsize]below:{$y^{(2)}$}}] at (3.05,0) {};
\foreach \p in {(3.12,0.12),(3.19,0.24),(3.26,0.36),(3.33,0.48)} \node[circle,fill=red!70!black,inner sep=0.9pt] at \p {};
\foreach \p in {(3.42,0.62),(3.52,0.78),(3.62,0.94),(3.72,1.10)} \node[circle,fill=orange!85!black,inner sep=0.9pt] at \p {};
\node[circle,fill=black,inner sep=1.6pt,label={[font=\scriptsize]right:{$y^{(3)}$}}] at (3.82,1.26) {};
\node[align=left,font=\scriptsize,anchor=west] at (-0.4,-1.05) {{\color{red!70!black}$\bullet$}\; additions of the new coordinate};
\node[align=left,font=\scriptsize,anchor=west] at (-0.4,-1.45) {{\color{orange!85!black}$\bullet$}\; balanced ordering toward the next interior point};
\end{tikzpicture}
\caption{The path built in Lemma~\ref{lem:integer_path}. Each time the support grows, finitely many copies of the new type first push the count vector clear of the old face; a balanced ordering then follows the segment to a large integer multiple of an interior point of the new face.}
\label{fig:seed_path}
\end{figure}

We now show that Condition D yields a {\em single} finite digraph $H$ on vertex set $[Q]$ whose subgraphs $H[k]:=H[\{1,\dots,k\}]$ induced by the first $k$ vertices all admit cycle covers, $1\leq k\leq Q$. Finding one copy of $H$ in $D_n$ therefore delivers a whole initial segment of the nested sequence at once.

\begin{lemma}[A finite digraph from Condition D]\label{lem:finite_prefix_template}
Assume Condition D holds. Then there exist $Q\in\N$, a point $y^\star\in\rint\vec{\mathcal X}(S)\cap\mathbb Q^q$ with full support, a directed graph $H$ on vertex set $[Q]$ and a map $\tau:[Q]\to U$ assigning a type to each vertex of $H$, such that: \textup{(i)} every edge $a\to b$ of $H$ is admissible, i.e., $(\tau(a)\to\tau(b))\in F$; \textup{(ii)} $H[k]$ admits a cycle cover for every $1\leq k\leq Q$; \textup{(iii) $H[1]$ is a single looped vertex; and \textup{(iv)} the type counts of $[Q]$ are $Qy^\star$, i.e., $|\tau^{-1}(u_i)|=Qy^\star_i$ for every $i$.}
\end{lemma}

\begin{proof}
By Proposition~\ref{prop:realizable_supports}, the support faces corresponding to the ordering of Condition D satisfy the hypotheses of Lemma~\ref{lem:integer_path} with $P=\vec{\mathcal X}(S)$. Let $m^{(1)},\dots,m^{(Q)}$ be the resulting path of integer points and define $\tau(k):=u_i$ when $m^{(k)}-m^{(k-1)}=e_i$, with the convention $m^{(0)}:=0$. For each $k$, the digraph on $[k]$ containing all admissible edges is isomorphic to $S[m^{(k)}]$, so it contains, owing to Proposition~\ref{prop:blowup_cover},  a cycle cover $C_k$. Put $E(H):=\bigcup_{k=1}^Q E(C_k)$; every edge of $C_k$ joins two vertices of $[k]$, so $H[k]$ contains $C_k$ and therefore admits a cycle cover, for every $k$. Since $m^{(1)}=e_{t_1}$ and $t_1$ has a loop by the case $r=1$ of Condition D, $C_1$ is a loop. Finally, $y^\star:=m^{(Q)}/Q$ belongs to  $\rint\vec{\mathcal X}(S)\cap\mathbb Q^q$, and it has full support because Condition D makes $\vec{\mathcal X}(S)$ full-dimensional in the simplex, so that no coordinate can vanish at a relative interior point.
\end{proof}

The following Corollary will be used to show that the conditions we exhibited are tight, and is essentially the converse of the previous Lemma.
\begin{corollary}\label{cor:D_blowup}
Condition D holds if and only if there exist $m\in\Z_{>0}^q$ and an ordering of the vertices of $S[m]$, beginning at a looped vertex, such that the subgraph induced by every prefix of this ordering admits a cycle cover.
\end{corollary}

\begin{proof}
If Condition D holds, take $m:=Qy^\star$ from Lemma~\ref{lem:finite_prefix_template}; the blow-up $S[m]$ contains $H$ on the same vertex set and with the same ordering, hence inherits the cycle covers of all its prefixes, and its first vertex is looped. Conversely, suppose such an ordered blow-up exists, and list the types of vertices in the order in which they first appear. Consider the prefix that ends at the first vertex of the $r$th type to appear. Being an induced subgraph of a complete blow-up, that prefix is itself a complete blow-up, so by  Proposition~\ref{prop:blowup_cover}, its type frequency vector is in $\vec{\mathcal X}(S)$; its support is exactly the set of types of vertices seen so far. These $q$ points certify Condition D by Proposition~\ref{prop:realizable_supports}\textup{(i)}, the first one forcing $t_1\in L$.
\end{proof}

\subsection{The matching lemma}\label{ssec:dense_block_matching}

We turn to the probabilistic half. The setting is a random bipartite graph whose two sides are each partitioned into $q$ blocks (one block per node of $S$) with edges allowed only between blocks joined by an edge of $S$ (equivalently, for which $p_{ij}>0$). If the block sizes satisfy Hall's inequality, and do so with a margin proportional to the total size, then a {\em perfect matching exists with probability exponentially close to $1$}.
\begin{lemma}[Perfect matchings in random block bipartite graphs]\label{lem:dense_block_matching_combined}
Assume that $S=(U,F)$ satisfies $N^+(U)=U$, and fix $\delta,p_0\in(0,1]$. Then there exist $c>0$ and $k_0\in\N$, depending only on $\delta$, $p_0$ and $q$, such that the following holds for every $k\geq k_0$. Let $m\in\Z_{\ge0}^q$ satisfy $\bfo^\top m=k$ and
\begin{equation}\label{eq:ineqlemhall}
m_I\leq m_{N^+(I)}-\delta k\qquad\text{for every }\varnothing\neq I\subsetneq U,
\end{equation} 
and let $B=(L\sqcup R,E)$ be a random bipartite graph with block decompositions $L=\bigsqcup_iL_i$ and $R=\bigsqcup_jR_j$ satisfying $|L_i|=|R_i|=m_i$, with no edges from $L_i$ to $R_j$ when $(u_i\to u_j)\notin F$ and with the admissible edges present independently, each with probability at least $p_0$. Then
$$
\Pr(B\text{ has a perfect matching})\ge1-e^{-ck}.
$$
\end{lemma}

The proof splits Hall obstructions into two categories. Small obstructions, of size at most $\varepsilon k$, are ruled out by the margin, which guarantees that linearly many right vertices outside the neighborhood of the obstruction could be joined to it; large obstructions, of linear size, are ruled out by a compactness argument showing that such an obstruction requires {\em quadratically} many admissible edges to be absent at once.

\begin{proof}

We first note  that the margin on right neighbors~\eqref{eq:ineqlemhall} in the statement implies the following margin on left neighbors:
$$
m_J\leq m_{N^-(J)}-\delta k\qquad(\varnothing\neq J\subsetneq U).
$$
Indeed, let $J$ be nonempty and proper. If $N^-(J)=U$, then~\eqref{eq:ineqlemhall}, applied to the subset $J$ of $U$, gives $m_J\leq m_{N^+(J)}-\delta k\leq k-\delta k=m_{N^-(J)}-\delta k$. Otherwise set $I:=U\setminus N^-(J)$, which is nonempty because $N^-(J)\neq U$, and proper  because $N^-(J)\neq\varnothing$: a node of $J$ with no incoming edge would not belong to $N^+(U)$. There is no edge of $S$ from $I$ to $J$, so $N^+(I)\subseteq U\setminus J$, whence $m_{N^+(I)}\leq k-m_J$. Since $m_I=k-m_{N^-(J)}$, applying~\eqref{eq:ineqlemhall} to $I$ gives $k-m_{N^-(J)}\leq k-m_J-\delta k$, which proves the claim.

Call a pair $(v,w)\in L_i\times R_j$ with $(u_i\to u_j)\in F$ an {\em admissible potential edge}: it is present in $B$ independently of the other pairs, with probability at least $p_0$. By Lemma~\ref{lem:small_hall_witness} it is enough to rule out left and right obstructions of size at most $\lceil k/2\rceil$. We treat the left ones; the right ones follow by the same argument on the transposed graph, using the margin on left neighbors. To this end,  let $X\subseteq L$ have size $s$ with $|N(X)|\leq s-1$, and choose $Y\subseteq R$ with $N(X)\subseteq Y$ and $|Y|=s-1$.

Fix a small $\varepsilon>0$, to be chosen below, and suppose first that $s\le\varepsilon k$. Let $I$ be the set of indices $i$ with $X\cap L_i\neq\varnothing$, so that $s\leq m_I$. If $I\neq U$, then
$$
m_{N^+(I)}-|Y|\geq m_I+\delta k-(s-1)\ge\delta k+1.
$$
If $I=U$, then $m_{N^+(I)}-|Y|=k-s+1\geq k/3$ for large $k$. Thus at least $\alpha k$ right vertices outside $Y$, where $\alpha:=\min\{\delta,1/3\}$, each have an admissible potential edge from $X$. Hence
$$
\Pr(N(X)\subseteq Y)\leq e^{-p_0\alpha k}.
$$
Choosing $\varepsilon$ so that $3\varepsilon\log(e/\varepsilon)<p_0\alpha/2$ and summing over all pairs $(X,Y)$ with $s\le\varepsilon k$ gives an upper bound $e^{-c_1k}$.

Now suppose instead that $\varepsilon k\leq s\le\lceil k/2\rceil$, so that the obstruction is of linear size. Set $x_i:=|X\cap L_i|$, $y_j:=|Y\cap R_j|$, and normalize
$$
\mu_i:=\frac{m_i}{k},\qquad a_i:=\frac{x_i}{k},\qquad b_j:=\frac{y_j}{k}.
$$
The number of admissible potential edges from $X$ to $R\setminus Y$, divided by $k^2$, is
$$
\Psi(\mu,a,b):=\sum_{(u_i\to u_j)\in F}a_i(\mu_j-b_j).
$$
The vectors $(\mu,a,b)$ range over the  compact set defined by the following inequalities: 
$$\mu\in\Delta_q,\quad \mu_I\leq \mu_{N^+(I)}-\delta\ \ (\varnothing\neq I\subsetneq U),\quad 0\leq a,b\leq \mu, \quad \varepsilon\le\bfo^\top a\le2/3, \quad \bfo^\top b\le\bfo^\top a.$$
On that set $\Psi$ is strictly positive. Indeed, if $\Psi=0$ and $I:=\{i:a_i>0\}$, then $b_j=\mu_j$ for every $j\in N^+(I)$. Consequently,
$$
\mu_{N^+(I)}=b_{N^+(I)}\le\bfo^\top b\le\bfo^\top a\leq \mu_I.
$$

For $I\neq U$ this contradicts the margin, and for $I=U$ it yields $1=\mu_{N^+(U)}\le2/3$, a contradiction. Since $\Psi$ is continuous and the set is compact and {\em independent of} $k$, we get $\Psi\ge\gamma>0$ uniformly. For fixed $X$ and $Y$, then, at least $\gamma k^2$ admissible potential edges must all be absent, so
$$
\Pr(N(X)\subseteq Y)\leq e^{-p_0\gamma k^2}.
$$
A union bound over the at most $4^k$ pairs $(X,Y)$ leaves $e^{-c_2k^2}$. Combining the four bounds (two regimes on each side) proves the lemma.
\end{proof}

 The following lemma provides the form of Lemma~\ref{lem:dense_block_matching_combined} that will be used in the proofs below.  In the sufficiency proofs below, we will exhibit nested sequences of vertex sets that depend on the sampled graph. For $\mathcal N$, the sequence is built from the type map $\pi$ alone. For $\mathcal S$ we split the edges of $D_n$ into {\em two independent layers} (Section~\ref{sec:sufficient_S}), we find the initial segment of the sequence using the first layer and then keep the second layer for the cycle covers, so that the sequence depends on $\pi$ together with the first layer. In both cases the data used to build the sequence is independent of the edges left over, and we write $\mathcal F$ for the $\sigma$-field they generate: given $\mathcal F$, the sequence and the types are fixed, while the remaining admissible edges are still independent. 

\begin{lemma}[Cycle covers along a fixed nested sequence]\label{lem:uniform_chain_realization}
Assume that Condition C holds. Let $K\Subset\rint\vec{\mathcal X}(S)$ and $p_0>0$. Then there exist $c,C>0$ and $k_0\in\N$ with the following property. Let $D$ be a random digraph whose vertices have types in $U$ and whose edges are all admissible, and let $\mathcal F$ be a $\sigma$-field for which the type map of $D$ and a nested sequence of vertex sets $J_M\subset\cdots\subset J_N$ with $|J_k|=k$ are measurable, and conditionally on which the admissible edges of $D$ are present independently, each with probability at least $p_0$. If $M\geq k_0$ and $x(J_k)\in K$ for every $M\leq k\leq N$, then
$$
\Pr\bigl(\exists k\in\{M,\dots,N\}:D[J_k]\text{ has no cycle cover}\mid\mathcal F\bigr)\leq Ce^{-cM}.
$$
\end{lemma}

\begin{proof}
We first produce a uniform margin on $K$. If $q=1$ there is nothing to prove and any $\delta>0$ will do. If $q\ge2$, fix a nonempty proper $I\subsetneq U$. The Hall functional $\Phi_I(x)=x_{N^+(I)}-x_I$ is nonnegative on $\vec{\mathcal X}(S)$ by Lemma~\ref{lem:hall_Xdir}, and it is not identically zero there: by Lemma~\ref{lem:proper_expand} the vector $\mathbf 1_{N^+(I)}-\mathbf 1_I$ is nonzero and is not a multiple of $\bfo$, while $\dim\vec{\mathcal X}(S)=q-1$ (Condition C) makes the affine hull of the polytope the whole hyperplane $\{\bfo^\top x=1\}$. A nonnegative linear functional that is not identically zero on a convex set is {\em strictly} positive on its relative interior, so $\Phi_I>0$ on $\rint\vec{\mathcal X}(S)$. Since $K$ is compact and there are finitely many subsets $I$, there is $\delta>0$ with $\Phi_I(x)\ge\delta$ for all $x\in K$ and all such $I$.

Now fix $k$ with $M\leq k\leq N$ and let $m$ be the vector of type counts of $J_k$. Conditional on $\mathcal F$, the bipartite lift of $D[J_k]$ is a random block bipartite graph with block sizes $m$, and these block sizes satisfy the hypothesis~\eqref{eq:ineqlemhall} of Lemma~\ref{lem:dense_block_matching_combined}: since $x(J_k)=m/k$ lies in $K$, the previous paragraph gives $\Phi_I(m/k)\ge\delta$ for every nonempty proper $I\subsetneq U$, which implies $m_I\leq m_{N^+(I)}-\delta k$. 

Note that the hypothesis $N^+(U)=U$ of Lemma~\ref{lem:dense_block_matching_combined} holds here: indeed, for $q\ge2$ it follows from strong connectivity of $S$, and for $q=1$, Condition C gives $\vec{\mathcal X}(S)\neq\varnothing$, which forces a loop at the unique node of $S$. Lemma~\ref{lem:dense_block_matching_combined}, applied with the  $\delta$ obtained above and with $p_0$, bounds the probability that the lift has no perfect matching by $e^{-ck}$, with $c$ and $k_0$ depending only on $\delta$, $p_0$ and $q$. By Lemma~\ref{lem:cyclecover_matching}, $D[J_k]$ has a cycle cover if and only if its bipartite lift has a perfect matching, so $\Pr\bigl(D[J_k]\text{ has no cycle cover}\mid\mathcal F\bigr)\leq e^{-ck}$ as well, and a union bound gives
$$
\sum_{k=M}^N e^{-ck}\leq Ce^{-cM},
$$ as claimed.
\end{proof}

\section{Necessary conditions}\label{sec:necessary}

We now prove the necessity halves of Theorems~\ref{thm:maincondN} and~\ref{thm:maincondS}, using the characterization of $\vec{\mathcal X}(S)$ given in  Section~\ref{sec:preliminaries}. 
%\subsection{Anti-concentration}\label{ssec:anti_conc_lemma}

 The next lemma is elementary; it shows that a multinomial random variable does not concentrate on any affine hyperplane, other than the hyperplane $\bfo^\top y=n$ on which it always lies.
\begin{lemma}\label{lem:anti_conc}
Let $N=(N_1,\dots,N_q)\sim\mathrm{Mult}(n;x^*)$ with $x_i^*>0$ for every $i$, and let $a\in\R^q$ not be a multiple of $\bfo$. Then there exists $C=C(a,x^*)<\infty$ such that
$$
\sup_{t\in\R}\ \Pr\bigl(a^\top N=t\bigr)\ \le\ \frac{C}{\sqrt n}.
$$
\end{lemma}

\begin{proof}
Since $a$ is not a multiple of $\bfo$, we may fix $i\neq j$ with $a_i\neq a_j$. Condition on the counts $(N_\ell)_{\ell\neq i,j}$. This determines $M:=N_i+N_j=n-\sum_{\ell\neq i,j}N_\ell$, and conditionally $N_i\sim\Bin(M,p)$ with $p:=x_i^*/(x_i^*+x_j^*)\in(0,1)$, by the standard conditioning property of the multinomial distribution. Letting $N_j=M-N_i$ in the linear form gives $a^\top N=c+(a_i-a_j)N_i$, where $c$ is fixed by the conditioning; since $a_i\neq a_j$, each $t\in\R$ leaves at most one possible value of $N_i$, and thus
$$
\Pr\bigl(a^\top N=t\ \big|\ (N_\ell)_{\ell\neq i,j}\bigr)\ \le\ \max_{r}\ \Pr\bigl(\Bin(M,p)=r\bigr)\ \le\ \frac{C_p}{\sqrt{M}},
$$ the last step being the standard bound on the largest atom of a binomial distribution with fixed success probability $p\in(0,1)$.

Finally, $M\sim\Bin(n,x_i^*+x_j^*)$ unconditionally, so by Hoeffding's inequality $\Pr\bigl(M<\tfrac12(x_i^*+x_j^*)n\bigr)\leq e^{-cn}$ for some $c>0$. Splitting on this event,
$$
\Pr\bigl(a^\top N=t\bigr)\ \le\ \frac{C_p}{\sqrt{\tfrac12(x_i^*+x_j^*)n}}+e^{-cn},
$$ uniformly in $t$, which proves the claim.
\end{proof}

\subsection{Three necessary conditions}\label{ssec:three_necessary}

We now address the  necessity part of both Theorems~\ref{thm:maincondN} and~\ref{thm:maincondS}: if either A, B or C is not met, then the probabilities of both $\mathcal N$ and $\mathcal S$ tend to zero asymptotically.

\begin{theorem}[Necessary conditions]\label{thm:necessary}
Let $W$ be a directed step-graphon with skeleton $S$ and concentration vector $x^*$. Let $D_n\sim W$. If either
\begin{enumerate}
\item $x^*\notin\vec{\mathcal X}(S)$, or
\item $\dim \vec{\mathcal X}(S)\leq q-2$, or
\item $\Lambda=\sum_i x_i^*p_{ii}=0$,
\end{enumerate} then $\lim_{n\to\infty}\Pr(\mathcal N)\ =\ \lim_{n\to\infty}\Pr(\mathcal S)\ =\ 0.$
\end{theorem}

\begin{proof}
Let $\mathcal E_n$ be the event that $D_n=(V,\vec E)$ admits an $n$-cycle cover (a {\em Hamiltonian decomposition} in the terminology of~\cite{belabbas2013sparse,belabbas2021h}). By Lemma~\ref{lem:cover_implies_Xdir} with $V'=V$, $\mathcal E_n$ implies $x(V)\in\vec{\mathcal X}(S)$. Since both $\mathcal N$ and $\mathcal S$ require an $n$-cycle cover, it suffices\red{\,}\blue{,} in Cases 1 and 2, to show $\Pr(\mathcal E_n)\to 0$.

\medskip
\xc{Case 1: $x^*\notin\vec{\mathcal X}(S)$.}
Since $\vec{\mathcal X}(S)$ is closed and $x^*\notin\vec{\mathcal X}(S)$, there exists an open neighborhood $\mathcal U$ of $x^*$ in $\R^q$ such that $\mathcal U\cap\vec{\mathcal X}(S)=\varnothing$. By Lemma~\ref{lem:multinomial_conc}, $\Pr(x\in\mathcal U)\to 1$, hence $\Pr(\mathcal E_n)\leq \Pr(x\in\mathcal U^c)\to 0$.

\medskip
\xc{Case 2: $\dim \vec{\mathcal X}(S)\leq q-2$.}
If $\vec{\mathcal X}(S)=\varnothing$, then Case~1 already applies, so assume $\vec{\mathcal X}(S)\neq\varnothing$. The affine hull $\aff(\vec{\mathcal X}(S))$ is then a proper affine subspace of the hyperplane $\{\bfo^\top x=1\}$, so it satisfies at least one additional constraint: there are $a\in\R^q$, not a multiple of $\bfo$, and $b\in\R$ with
$$
a^\top x=b\qquad\text{for every }x\in\vec{\mathcal X}(S).
$$
Let $N=nx$ denote the vector of type counts, which is $\mathrm{Mult}(n;x^*)$  by Lemma~\ref{lem:multinomial_conc}, with $x^*_i>0$ for every $i$.  Consequently by Lemma~\ref{lem:anti_conc} we have 
$$
\Pr(\mathcal E_n)\ \le\ \Pr\bigl(a^\top N=nb\bigr)\ \le\ \frac{C}{\sqrt n}\ \longrightarrow\ 0.
$$

\medskip
\xc{Case 3: $\Lambda =0$.} If $\Lambda=0$, then $L_n=0$ almost surely by Lemma~\ref{lem:binomial_loops}. Hence $D_n$ has no loop and therefore no $1$-cycle cover; both $\mathcal N$ and $\mathcal S$ fail almost surely.
\end{proof}

\subsection{The boundary case: $x^*\in\partial\vec{\mathcal X}(S)$}\label{ssec:boundary_case}

We now show that if Condition B holds but $x^*$ lies on the {\em boundary} of the cycle polytope, then $\Pr(\mathcal N)$ cannot tend to $1$. This complements Theorem~\ref{thm:necessary} and explains the gap between B and $\mathrm{B}^*$ in our main theorems. See~\cite{gao2025h} for a more detailed analysis of the symmetric case.

\begin{proposition}\label{prop:rint_necessary_prob1}
Assume Condition C holds. If $x^*\in\vec{\mathcal X}(S)\setminus\rint(\vec{\mathcal X}(S))$, then
$$
\limsup_{n\to\infty}\Pr\bigl(D_n\text{ admits an $n$-cycle cover}\bigr)\leq \frac{1}{2}.
$$
\end{proposition}

\begin{proof}
Since $x^*\in\vec{\mathcal X}(S)\setminus\rint(\vec{\mathcal X}(S))$ and $\dim \vec{\mathcal X}(S)=q-1$, the point $x^*$ lies on the boundary of a full-dimensional polytope inside $\{x:\bfo^\top x=1\}$. Moreover, $x_i^*>0$ for all $i\in[q]$ by Definition~\ref{def:dir_stepgraphon}, so $x^*$ lies on no coordinate facet $\{x_i=0\}$. By Lemma~\ref{lem:hall_Xdir}, $\vec{\mathcal X}(S)$ is cut out of $\{\bfo^\top x=1\}$ by the inequalities $x\ge0$ and $x_I\leq x_{N^+(I)}$, so a boundary point satisfies one of them with equality. Hence there exists a nonempty proper subset $\varnothing\neq I\subsetneq U$ with $N^+(I)\neq I$ (by Lemma~\ref{lem:proper_expand}) such that $x^*_I = x^*_{N^+(I)}$.

Let
$$
\alpha_i = \mathbf 1_{\{i\in I\}} - \mathbf 1_{\{i\in N^+(I)\}},
\qquad
N=(N_1,\dots,N_q)\sim \mathrm{Mult}(n\red{,}\blue{;}x^*),
$$ and define $Z_n:= N_I - N_{N^+(I)}=\alpha^\top N.$ Since $N^+(I)\neq I$, the vector $\alpha$ is not a multiple of $\bfo$.

If $D_n$ admits an $n$-cycle cover, then $x(V)\in\vec{\mathcal X}(S)$ by Lemma~\ref{lem:cover_implies_Xdir}. Hence Lemma~\ref{lem:hall_Xdir} gives $N_I\leq N_{N^+(I)}$ and thus $Z_n\leq 0$. Therefore $$\Pr(\text{$D_n$ admits an $n$-cycle cover})\leq \Pr(Z_n\leq 0).$$

We claim that the distribution of $Z_n$ is symmetric about zero. To see this, let $\xi_\ell \in \{e_1,\ldots,e_q\}$, $1 \leq \ell \leq n$ be a random vector with $\xi_\ell = e_i$ if the $\ell$th draw in $N$ is a vertex of type $i$. Then, setting $Y_\ell:=\alpha^\top \xi_\ell$, we can write $Z_n=\sum_{\ell=1}^n Y_\ell$. Since $\alpha_i\in\{-1,0,1\}$, we have $Y_\ell\in\{-1,0,1\}$. Also, $\E Y_\ell=\alpha^\top x^*=x_I^*-x_{N^+(I)}^*=0,$ so $\Pr(Y_\ell=1)=\Pr(Y_\ell=-1).$ Hence each $Y_\ell$ is symmetric about $0$, and therefore so is $Z_n$; this proves the claim.

Appealing to Lemma~\ref{lem:anti_conc},
$$
\Pr(Z_n=0)\leq \frac{C}{\sqrt n}.
$$
Thus, using symmetry of $Z_n$ for the second equality, we get
$$
\Pr(Z_n\leq 0)= \Pr(Z_n <0)+\Pr(Z_n=0)=\frac{1}{2}+\frac{1}{2}\Pr(Z_n=0)
\leq \frac{1}{2}+O(n^{-1/2}).
$$
Taking $\limsup$ concludes the proof.
\end{proof}

\subsection{Necessity of Condition D}\label{ssec:necessity_D}

Condition D stems from a structural feature that any of the nested sequences certifying $\mathcal S$ displays: vertices are added one at a time, and the support of the type frequency vector grows monotonically, by at most one node at a time. Keeping track of this growth gives the ordering required by Condition D.

\begin{proposition}[Necessity of Condition D]
\label{prop:Dadm_necessary} Let $W$ be a directed step-graphon with strongly connected skeleton $S=(U,F)$ and concentration vector $x^*\in\Delta_q$. If $\limsup_{n\to\infty}\Pr(\mathcal S)>0$, then Condition D holds.
\end{proposition}

\begin{proof}
Let $N_i$ denote the number of vertices of type $u_i$ in $D_n\sim W$. Then $N_i\sim \mathrm{Bin}(n,x_i^*)$. Hence $\Pr(N_i=0)=(1-x_i^*)^n\to 0$ for every $1 \leq i \leq q$ and using a union bound,
$$
\Pr(N_i\geq 1\ \forall i\in[q])\to 1.
$$

Now assume $\limsup_{n\to\infty}\Pr(\mathcal S)>0$. Then there exists a subsequence $n_j\to\infty$ and a constant $\eta>0$ such that $ \Pr(D_{n_j}\text{ has property }\mathcal S)\geq \eta$ for all $j$. Combined with the previous equation, this gives, for all sufficiently large $j$,
$$
\Pr\bigl(\{D_{n_j}\text{ has property }\mathcal S\}\ \cap\ \{N_i(D_{n_j})\geq 1\ \forall i\in[q]\}\bigr)>0.
$$
Fix such a $j$ and a realization of $D_{n_j}$ in this event; in particular, vertices of every type appear in $D_{n_j}$. Let $J_1\subset J_2\subset\cdots\subset J_{n_j}=V(D_{n_j})$, with $ |J_k|=k$, be a nested sequence certifying property $\mathcal S$. Then it satisfies
\begin{enumerate}
\item $J_1$ is a single looped vertex;
\item $D_{n_j}[J_k]$ admits a cycle cover for every $k=1,\dots,n_j$.
\end{enumerate}

For each $k$, define $T(k):=\pi(J_k)\subseteq U.$ Because $J_{k+1}$ is obtained from $J_k$ by adding exactly one vertex, the support can increase by {\em at most} one node at a time: $|T(k+1)\setminus T(k)|\leq 1$ for all $k$. With a slight abuse of notation, we only keep in the sequence the distinct supports as they first appear: this gives a nested sequence $T_1\subset T_2\subset\cdots\subset T_m=U.$ Since $J_1$ consists of a single vertex, $T_1=\{t_1\}$ for some $t_1 \in L$; since supports increase by at most one coordinate at a time, $|T_r|=r$ for $r=1,\dots,m$; and since vertices of all types appear, $m=q$. Define the corresponding ordering of the nodes of $S$, namely $t_1,\ldots, t_q$ so that $ T_r=\{t_1,\dots,t_r\}$ for $ 1 \leq r \leq q$.

For each $r=1,\dots,q$, let $k_r$ be the first index at which the support $T_r$ appears: $T(k_r)=T_r$. Because $D_{n_j}[J_{k_r}]$ admits a cycle cover, Lemma~\ref{lem:cover_implies_Xdir} implies that $x(J_{k_r})\in \vec{\mathcal X}(S)$. Moreover, by definition of $T_r$,
$$
\supp\bigl(x(J_{k_r})\bigr)=\pi(J_{k_r})=T_r.
$$
Thus each $T_r$ is the support of a point of $\vec{\mathcal X}(S)$, so by Proposition~\ref{prop:realizable_supports}\textup{(i)} the ordering $t_1,\dots,t_q$ satisfies Condition D.
\end{proof}

\section{Sufficiency for $\mathcal N$}\label{sec:sufficient_N}
We now prove that Conditions A, $\mathrm{B}^*$ and C are sufficient for $\mathcal N$ to hold.  In the proof, the cycle covers on a small number of vertices are provided by loops and those on many vertices by unions of disjoint cycles of arbitrary lengths;  the cutoff between the two regimes is taken to be a fixed fraction of $n$.

\begin{theorem}[Sufficient conditions for $\mathcal N$]\label{thm:suff_N}
Let $W$ be a directed step-graphon with strongly connected skeleton $S$, concentration vector $x^*$, and block probabilities $(p_{ij})$. Assume Conditions A, ${B}^*$ and C hold. Then there exist constants $C_W,c_W>0$, depending only on $W$, such that $\Pr(\mathcal N)\ge1-C_We^{-c_Wn}$ for $D_n\sim W$ and all $n$; in particular, $\Pr(\mathcal N)\to 1$.
\end{theorem}

\begin{proof}
From Condition C, the affine hull of $\vec{\mathcal X}(S)$ is equal to $\{\bfo^\top x=1\}$, so owing to Condition $\mathrm{B}^*$ we can choose $\varepsilon>0$ so that
$$
K:=\{y\in\Delta_q:\|y-x^*\|_\infty\le\varepsilon\}\Subset\rint\vec{\mathcal X}(S).
$$
Set $M:=\lfloor\Lambda n/2\rfloor$. From Condition A, we have  $\Lambda>0$, so by Lemma~\ref{lem:binomial_loops} and Hoeffding's inequality the number of looped vertices of $D_n$ is at least $M$ with probability at least $1-e^{-\Lambda^2n/2}$; on that event any $k$ of them  define an induced subgraph covered by $k$ loops, for every $1\leq k\leq M$.

For the remaining sizes $k>M$, we use Lemma~\ref{lem:balanced_prefix_rounding_aux} on the type counts $N=(N_1,\dots,N_q)$ to obtain an ordering $v_1,\dots,v_n$ of the vertices that depends only on the type map, whose prefixes $V_k:=\{v_1,\dots,v_k\}$ satisfy
$$
\|x(V_k)-x\|_\infty\le\frac{C_{\mathrm{bal}}(q)}{k}.
$$
By Lemma~\ref{lem:multinomial_conc} we have $\|x-x^*\|_\infty\le\varepsilon/2$ with probability at least $1-2qe^{-c_1n}$ for some $c_1>0$, and $C_{\mathrm{bal}}(q)/M\le\varepsilon/2$ for $n$  large. Let $\mathcal G$ be the former event, $n$ being large enough for the latter inequality; it is measurable with respect to the type map $\pi$, it has probability at least $1-2qe^{-c_1n}$, and on this event
$$
x(V_k)\in K\qquad(M\leq k\leq n).
$$
Now we condition on $\pi$ and work on the event $\mathcal G$. The sequence $(V_k)_{k\geq M}$ is then fixed, while every admissible edge of $D_n$ is  present independently with probability at least $p_{\min}$. Then, from Lemma~\ref{lem:uniform_chain_realization} we have
$$
\Pr\bigl(\exists k\geq M:D_n[V_k]\text{ has no cycle cover}\mid\pi\bigr)\leq Ce^{-cM}\leq Ce^{c}e^{-c\Lambda n/2}
$$ on $\mathcal G$, for $n$ large enough that $M\geq k_0$.

Putting the two parts together, we obtain, for every $k\in\{1,\dots,n\}$, a $k$-vertex induced subgraph of $D_n$ that admits a cycle cover,  with probability at least $1-e^{-\Lambda^2n/2}-2qe^{-c_1n}-Ce^{c}e^{-c\Lambda n/2}$ for $n$ large. This is the claimed bound, after enlarging $C_W$ to account for small $n$.
\end{proof}

Theorem~\ref{thm:maincondN} follows from Theorems~\ref{thm:necessary} and~\ref{thm:suff_N}.

\section{Sufficiency for $\mathcal S$}\label{sec:sufficient_S}

Property $\mathcal S$ requires a nested sequence of vertex sets whose induced subgraphs all have a cycle cover. Lemma~\ref{lem:finite_prefix_template} provides the beginning of the sequence, and Lemma~\ref{lem:uniform_chain_realization} the remainder. Each of these being probabilistic, we must ensure that they do not rely on the same edges, so that we can consider both events independently. To this end, we split the edge probabilities in two. For each ordered pair $(i,j)\in[q]^2$, set
$$
p^{(0)}_{ij}:=\frac{p_{ij}}{2},\qquad p^{(1)}_{ij}:=\frac{p_{ij}}{2-p_{ij}},
$$ 
and generate two random digraphs $D^{(0)}$ and $D^{(1)}$, the {\em layers}, on the vertex set of $D_n$ and with the same type map $\pi$: conditionally on $\pi$, each ordered pair of vertices of types $u_i$ and $u_j$ is an edge of $D^{(\ell)}$ with probability $p^{(\ell)}_{ij}$, independently over pairs and over $\ell\in\{0,1\}$. Since $(1-p_{ij})=(1-p^{(0)}_{ij})(1-p^{(1)}_{ij})$, {\em their union has the same law as $D_n$}. We define $\mathcal H_0:=\sigma(\pi,D^{(0)})$ and use the first layer for the first part of the sequence, and the second layer for the remainder. 

\begin{lemma}\label{lem:log_seed_from_finite_seed_path}
Assume Condition D holds, and let $(H,\tau)$, $Q$, and $y^\star$ be as in Lemma~\ref{lem:finite_prefix_template}. There exist $\beta,c_2>0$, depending only on $W$ and $(H,\tau)$, such that for all sufficiently large $n$, setting $b:=\lfloor\beta n\rfloor$ and $M:=Qb$, with probability at least $1-e^{-c_2n}$ there is an $\mathcal H_0$-measurable nested sequence
$$
J_1\subset\cdots\subset J_M,\qquad |J_k|=k,
$$ 
such that $J_1$ is a single looped vertex, every $D^{(0)}[J_k]$ admits a cycle cover, and $x(J_M)=y^\star$.
\end{lemma}

\begin{proof}
Let $d_i:=|\tau^{-1}(u_i)|=Qy_i^\star>0$. Every coordinate of $x^*$ is positive by Definition~\ref{def:dir_stepgraphon}, so Lemma~\ref{lem:multinomial_conc} gives $N_i\geq x_i^*n/2$ for all $i$ with probability at least $1-2qe^{-c_1n}$ for some $c_1>0$. On this event, define
$$
\nu:=\left\lfloor\min_i\frac{N_i}{d_i}\right\rfloor\ge\lfloor\gamma n\rfloor,\qquad \gamma:=\min_i\frac{x_i^*}{2d_i}>0,
$$
% [orange comment addressed: "what is a type tuple? never defined; don't make up compound words randomly"]
and choose $\nu$ pairwise disjoint $Q$-tuples $\mathbf v^{(s)}=(v^{(s)}_1,\dots,v^{(s)}_Q)$, $1\leq s\le\nu$, of distinct vertices of $D_n$ such that $\pi(v^{(s)}_a)=\tau(a)$ for every $a\in[Q]$; that is, the $a$th vertex of each tuple has the type that $\tau$ assigns to the vertex $a$ of $H$. This is possible because there are $N_i\ge\nu d_i$ vertices of each type $u_i$. We make this choice as a function of $\pi$ only, for instance by taking the vertices of each type in increasing order of their index.
We say that $\mathbf v^{(s)}$ {\em realizes} $H$ in $D^{(0)}$ if $v^{(s)}_a\to v^{(s)}_c$ is an edge of $D^{(0)}$ for every edge $a\to c$ of $H$, loops included; in other words, if $a\mapsto v^{(s)}_a$ maps $H$ onto a (not necessarily induced) subgraph of $D^{(0)}$, preserving types. Conditional on $\pi$, each of these edges is present independently with probability $p_{\tau(a)\tau(c)}/2$, so each tuple realizes $H$ in $D^{(0)}$ with the fixed probability
$$
p_{\mathrm{temp}}:=\prod_{(a,c)\in E(H)}\frac{p_{\tau(a)\tau(c)}}{2}>0.
$$
Conditional on the type map, these events are independent: the event that $\mathbf v^{(s)}$ realizes $H$ depends only on the edges of $D^{(0)}$ between vertices of $\mathbf v^{(s)}$; since the tuples are pairwise disjoint, distinct tuples involve disjoint sets of ordered pairs of vertices; and, conditional on $\pi$, the edges of $D^{(0)}$ on distinct ordered pairs are independent (Definition~\ref{def:sampling}). Hence, conditional on $\pi$, the number $Z$ of tuples that realize $H$ is binomial with parameters $\nu$ and $p_{\mathrm{temp}}$, and its mean $\nu p_{\mathrm{temp}}$ is at least $\lfloor\gamma n\rfloor p_{\mathrm{temp}}$. Set $\beta:=\gamma p_{\mathrm{temp}}/4$ and $b:=\lfloor\beta n\rfloor$. The Chernoff bound $\Pr(Z<\nu p_{\mathrm{temp}}/2)\leq e^{-\nu p_{\mathrm{temp}}/8}$, together with $\nu p_{\mathrm{temp}}/2\geq b$, which holds once $\gamma n\ge2$, shows that at least $b$ tuples realize $H$ with probability at least $1-e^{-\lfloor\gamma n\rfloor p_{\mathrm{temp}}/8}$.

Let $\mathbf v^{(s_1)},\dots,\mathbf v^{(s_b)}$, with $s_1<\dots<s_b$, be the first $b$ tuples that realize $H$, and list their vertices tuple after tuple, the vertices of each tuple in the order given by the labels of $H$:
$$
v^{(s_1)}_1,\dots,v^{(s_1)}_Q,\ v^{(s_2)}_1,\dots,v^{(s_2)}_Q,\ \dots,\ v^{(s_b)}_1,\dots,v^{(s_b)}_Q.
$$
For $1\leq k\leq M=bQ$, let $J_k$ be the set of the first $k$ vertices of this list. Listing each tuple in the order of the labels of $H$ is what makes every prefix of a tuple correspond to a set $[r]$ on which $H$ has a cycle cover. Any prefix of this list consists of several complete copies of $H$ followed by a prefix of the next copy; the union of the covers of those copies with the cover of that prefix is a cycle cover of that prefix, and all the edges involved are present in $D^{(0)}$ because each tuple realizes every edge of $H$. The first vertex, $v^{(s_1)}_1$, has a loop in $D^{(0)}$, because $H[1]$ is a looped vertex, and the $b$ complete copies together have type counts $bQy^\star$, so $x(J_M)=y^\star$. Altogether, the construction succeeds with probability at least $1-2qe^{-c_1n}-e^{-\lfloor\gamma n\rfloor p_{\mathrm{temp}}/8}\ge1-e^{-c_2n}$ for a suitable $c_2>0$ and all sufficiently large $n$.
\end{proof}

\begin{theorem}[Sufficiency for $\mathcal S$] \label{thm:sufficiency_Dslack}
Let $W$ be a directed step-graphon with strongly connected skeleton $S$ and concentration vector $x^*$. If Conditions $B^*$ and D hold, then there exist constants $C_W,c_W>0$, depending only on $W$, such that $\Pr(\mathcal S)\ge1-C_We^{-c_Wn}$ for all $n$; in particular, $\Pr(\mathcal S)\to 1$.
\end{theorem}

\begin{proof}
Let $y^\star$, $Q$ and $(H,\tau)$ be given by Lemma~\ref{lem:finite_prefix_template}. By Lemma~\ref{lem:log_seed_from_finite_seed_path}, with probability at least $1-e^{-c_2n}$ there is an $\mathcal H_0$-measurable nested sequence $J_1\subset\cdots\subset J_M$ with $M=Q\lfloor\beta n\rfloor$, in which $J_1$ is a single looped vertex, every $D^{(0)}[J_k]$ has a cycle cover, and $x(J_M)=y^\star$.

Both $y^\star$ and $x^*$ lie in $\rint\vec{\mathcal X}(S)$, the first by Lemma~\ref{lem:finite_prefix_template} and the second by $\mathrm{B}^*$, and the relative interior is convex; so the segment joining them is a compact subset of it, and we may choose $\rho>0$ for which
$$
K:=\{z\in H_U:\dist_\infty(z,[y^\star,x^*])\le\rho\}
$$ 
is still contained in $\rint\vec{\mathcal X}(S)$. This  compact set will contain every $x(J_k)$, $M\leq k\leq n$.

Let $d$ be the vector of type counts of $V\setminus J_M$, set $R:=n-M$, and let $x^{\mathrm{rem}}:=d/R$ be the corresponding type frequency vector. Since $J_M$ has type counts $My^\star$,
$$
x^{\mathrm{rem}}=\frac{nx-My^\star}{n-M},
$$ 
so that $x^{\mathrm{rem}}$ lies on the line through $y^\star$ and $x$.

Apply Lemma~\ref{lem:balanced_prefix_rounding_aux} to $d$, append the remaining vertices one at a time, the $s$th appended vertex being of type $u_i$ when $u^{(s)}-u^{(s-1)}=e_i$, and extend the sequence to $J_n=V$. Writing $u^{(s)}$ for the counts of the first $s$ appended vertices, we have
$$
x(J_{M+s})=\frac{M}{M+s}y^\star+\frac{s}{M+s}x^{\mathrm{rem}}+\varepsilon^{(s)} \mbox{ with } \|\varepsilon^{(s)}\|_\infty\le\frac{C_{\mathrm{bal}}(q)}{M+s}.
$$
Substituting the expression for $x^{\mathrm{rem}}$, the first two terms equal $\lambda_sy^\star+(1-\lambda_s)x$ with
$$
\lambda_s:=\frac{M(n-M-s)}{(n-M)(M+s)}\in[0,1],
$$
a point of the segment $[y^\star,x]$, which is within $\|x-x^*\|_\infty$ of the segment $[y^\star,x^*]$; the third term is at most $C_{\mathrm{bal}}(q)/M$ in norm. By Lemma~\ref{lem:multinomial_conc}, $\|x-x^*\|_\infty\le\rho/2$ with probability at least $1-2qe^{-c_3n}$ for some $c_3>0$, and $C_{\mathrm{bal}}(q)/M\le\rho/2$ for $n$ large, since $M$ grows linearly with $n$. On this event every $x(J_{M+s})$, $0\leq s\leq n-M$, lies in $K$.

Now, condition on $\mathcal H_0$; then the extended sequence is fixed, and the admissible edges of the second layer are independent with probabilities at least $p_{\min}/2$. As explained in Remark~\ref{rem:DimpliesCA},  Condition D implies Condition C, so Lemma~\ref{lem:uniform_chain_realization} applies and yields
$$
\Pr\bigl(\exists k\geq M:D^{(1)}[J_k]\text{ has no cycle cover}\mid\mathcal H_0\bigr)\leq Ce^{-cM}.
$$
Since both layers are subgraphs of $D_n$, the first layer covers the prefixes with $k\leq M$ and the second those with $k\geq M$, and the two families form one nested sequence from a single looped vertex up to $V$. Hence $1-\Pr(\mathcal S)\leq e^{-c_2n}+2qe^{-c_3n}+Ce^{-cM}$ for $n$ large, which is exponentially small in $n$ since $M\ge\beta Qn-Q$; potentially enlarging $C_W$ to account for small $n$ proves the claim.
\end{proof}

Theorem~\ref{thm:maincondS} follows from Theorem~\ref{thm:necessary}, Proposition~\ref{prop:Dadm_necessary}, and Theorem~\ref{thm:sufficiency_Dslack}.

\section{Conclusion and open problems}\label{sec:conclusion}

We have studied the structural stability of linear systems whose interconnection graph is sampled from a directed step-graphon $W$ with strongly connected skeleton $S$ and concentration vector $x^*$. The central object is the cycle polytope $\vec{\mathcal X}(S)$, the directed counterpart of the edge polytope governing the symmetric case. Theorems~\ref{thm:maincondN} and~\ref{thm:maincondS} give matching necessary and sufficient conditions, up to the boundary gap between B and $\mathrm{B}^*$, for the properties $\mathcal N$ and $\mathcal S$. The proofs rest on two facts. On the deterministic side, a vector of type counts is feasible exactly when the corresponding blow-up of the skeleton splits into disjoint cycles, which is what turns Condition D into a finite list of counts reachable one vertex at a time. On the probabilistic side, any nested sequence of vertex sets whose type frequency vectors stay uniformly inside the polytope acquires all of its cycle covers at once, with probability exponentially close to $1$.

Several questions remain open. The first is the boundary case where Condition B holds but $\mathrm{B}^*$ fails, i.e., $x^*$ lies on the boundary of the cycle polytope. Proposition~\ref{prop:rint_necessary_prob1} shows only that $\limsup_n\Pr(\mathcal N)\leq 1/2$ in this regime, leaving open the precise asymptotic; we expect the answer to depend on which face of $\vec{\mathcal X}(S)$ contains $x^*$, in analogy with the symmetric results of~\cite{gao2025h}. A second direction concerns the gap between $\mathcal N$, $\mathcal S$, and stability itself: even deterministically, no graph-theoretic characterization of stable digraphs is known, and our results show that $\mathcal N$ and $\mathcal S$ separate on the directed step-graphon model (Example~\ref{ex:NS_separation}), mirroring the deterministic gap between them. Finally, the cycle polytope is purely combinatorial and depends only on the skeleton $S$, not on the specific block probabilities $p_{ij}$ as long as they are positive on $F$; determining the optimal exponential rates in Theorems~\ref{thm:maincondN} and~\ref{thm:maincondS}, and tracking their dependence on $p_{\min}$, would be a natural refinement, as would extending the framework beyond step-graphons to general (non-block-structured) graphons, where the type space becomes continuous and the cycle polytope is replaced by an infinite-dimensional analog.

\printbibliography

\end{document}